\documentclass[12pt]{amsart}
\usepackage{latexsym}
\usepackage{float}
\usepackage{amsmath,amscd}
\usepackage{amssymb}
\usepackage{indentfirst}
\usepackage{amsthm}
\usepackage{thmtools}
\usepackage[capitalise]{cleveref}
\usepackage{graphicx}
\usepackage{amsfonts}
\usepackage{mathrsfs}
\usepackage{graphicx}
\usepackage{booktabs}
\usepackage{latexsym}
\usepackage{bm}
\usepackage{dsfont}
\usepackage{array}
\usepackage{color}
\usepackage{epstopdf}
\usepackage{subfigure}
\usepackage[top=1in, bottom=1in, left=1.25in, right=1.25in]{geometry}
\usepackage{mathtools}
\usepackage{leftidx}
\usepackage{tensor}
\usepackage{tikz}
\usepackage{tikz-cd}
\usetikzlibrary{matrix,arrows,decorations.pathmorphing}

\DeclareMathOperator{\Spec}{Spec}

\DeclareMathOperator{\ProjSym}{ProjSym}

\DeclareMathOperator{\Pic}{Pic}

\DeclareMathOperator{\im}{im}
\DeclareMathOperator{\mult}{mult}

\DeclareMathOperator{\rank}{rank}

\DeclareMathOperator{\Fr}{Fr}

\DeclareMathOperator{\Mat}{Mat}
\DeclareMathOperator{\Sht}{Sht}

\DeclareMathOperator{\all}{all}

\DeclareMathOperator{\I}{I}
\DeclareMathOperator{\II}{II}
\DeclareMathOperator{\Id}{Id}

\DeclareMathOperator{\Isom}{Isom}

\DeclareMathOperator{\Aut}{Aut}

\DeclareMathOperator{\SHom}{\mathscr{H}\kern -2pt \textit{om}}

\DeclareMathOperator{\ArtAlg}{ArtAlg}
\DeclareMathOperator{\Sets}{Sets}

\DeclareMathOperator{\bcirc}{\mathbin{\circ}}
\DeclareMathOperator{\simto}{\xrightarrow{\sim}}

\theoremstyle{definition}
\declaretheorem[parent=subsection]{Definition}

\theoremstyle{remark}
\declaretheorem[sibling=Definition]{Remark}
\theoremstyle{plain}
\declaretheorem[sibling=Definition]{Theorem}
\declaretheorem[sibling=Definition]{Lemma}
\declaretheorem[sibling=Definition]{Proposition}
\declaretheorem[sibling=Definition]{Corollary}

\crefname{Definition}{Definition}{Definitions}
\crefname{Remark}{Remark}{Remarks}
\crefname{Example}{Example}{Examples}
\crefname{Lemma}{Lemma}{Lemmas}
\crefname{Theorem}{Theorem}{Theorems}
\crefname{Proposition}{Proposition}{Propositions}
\crefname{Corollary}{Corollary}{Corollaries}
\crefname{Conjecture}{Conjecture}{Conjectures}
\crefname{Section}{Section}{Sections}
\crefname{Chapter}{Chapter}{Chapters}

\numberwithin{equation}{section}

\title[Rational functions on the moduli stack of Drinfeld shtukas]{Construction of certain rational functions on the moduli stack of Drinfeld shtukas}
\author{Zhiyuan Ding}

\begin{document}
\begin{abstract}
  We construct certain rational functions (modular units) on the moduli stack of Drinfeld shtukas. The divisors of these rational functions are supported on horospherical divisors of the moduli stack. The key to our construction is a vanishing theorem for shtukas with zeros and poles satisfying certain conditions. Using deformation theory, we calculate the divisors of these rational functions.
\end{abstract}

\maketitle

\section{Introduction}
\subsection{Conventions and notation}
The following conventions and notation will be used throughout the article.

Fix a prime number $p$ and fix a finite field $\mathbb{F}_q$ of characteristic $p$ with $q$ elements.

Let $X$ be a smooth projective geometrically connected curve over $\mathbb{F}_q$. Let $k$ be the field of rational functions on $X$. Let $\mathbb{A}$ be the ring of adeles of $k$. Let $O$ be the subring of integral adeles in $\mathbb{A}$.

For any scheme $S$ over $\mathbb{F}_q$, denote $\Phi_S=\Id_X\times \Fr_S:X\times S\to X\times S$, and let $\pi_S:X\times S\to S$ be the projection. We sometimes write $\Phi$ and $\pi$ instead of $\Phi_S$ and $\pi_S$ when there is no ambiguity about $S$.

For any scheme $S$ over $\mathbb{F}_q$, we denote $\Fr_S$ to be its Frobenius endomorphism relative to $\mathbb{F}_q$. For two schemes $S_1$ and $S_2$ over $\mathbb{F}_q$, $S_1\times S_2$ denotes the fiber product of $S_1$ and $S_2$ over $\mathbb{F}_q$, and by a morphism $S_1\to S_2$ we mean a morphism over $\mathbb{F}_q$.

For a scheme $S$ over $\mathbb{F}_q$ and two morphisms $\alpha,\beta:S\to X$, we say that they satisfy condition (\ref{Frobenius shifts of zero and pole mutually disjoint}) if
\begin{equation}\label{Frobenius shifts of zero and pole mutually disjoint}
\text{$\Gamma_{\Fr_X^i\bcirc\alpha},\Gamma_{\Fr_X^j\bcirc\beta},(i,j\in \mathbb{Z}_{\ge 0})$ are mutually disjoint subsets of $X\times S$.} \tag{$\ast$}
\end{equation}
We say that they satisfy condition (\ref{zero and pole not related by Frobenius}) if
\begin{equation}\label{zero and pole not related by Frobenius}
\text{$\Gamma_{\Fr_X^i\bcirc\alpha}\cap\Gamma_{\Fr_X^j\bcirc\beta}=\emptyset$ for all $i,j\in \mathbb{Z}_{\ge 0}$.} \tag{$+$}
\end{equation}
Condition (\ref{Frobenius shifts of zero and pole mutually disjoint}) is equivalent to the combination of (\ref{zero and pole not related by Frobenius}) and the condition that $\alpha$ and $\beta$ map $S$ to the generic point of $X$.

For a morphism $\alpha:S\to X$, we sometimes write $\alpha$ instead of $\Gamma_\alpha$.

\subsection{Background and motivation}
The notion of a shtuka (originally called an $F$-sheaf) was defined by Drinfeld in~\cite{Dr78,Drinfeld1}, and was used to prove the Langlands correspondence over function fields by Drinfeld in the case of $GL_2$ and by L. Lafforgue in the case of $GL_d$ for arbitrary $d\ge 2$. More recently, V. Lafforgue proved the automorphic-to-Galois direction of the Langlands correspondence for arbitrary reductive groups over function fields using a more general notion of shtukas.

In Section 6 of~\cite{Drinfeld3}, Drinfeld calculated the kernel of the intersection pairing for horospherical divisors of the (compactified) moduli scheme of shtukas of rank 2 with all level structures. Since $H^1$ of the moduli scheme vanishes as shown in~\cite{Drinfeld2}, this kernel equals the space of those horospherical divisors that are principal up to torsion. Its description can be thought of as a function field analogue of the Manin-Drinfeld theorem on cusps of modular curves. In Remark (2) of Section 6 of~\cite{Drinfeld3}, Drinfeld asked to explicitly construct the rational functions corresponding to those principal horospherical divisors.

In this article we construct certain rational functions on the moduli stack of Drinfeld shtukas of rank $\ge 2$. The divisors of these rational functions are supported on the horospherical divisors. The approach via toy shtukas in~\cite{Ding18} provides more principal horospherical divisors than the ones constructed here.

\subsection{Contents}
We give a review of Drinfeld shtukas in Section 1. In particular, we recall the notion of partial Frobeniuses, which is important in the theory of shtukas and will be used extensively throughout the article.

In Section 2, we introduce the notion of reducible shtukas. \cref{maximal trivial sub and maximal trivial quotient} plays an important role in our study of cohomology of shtukas.

The key to our construction of rational functions is given in Section 3. We prove several vanishing results (\cref{cohomology of shtukas over a field}, \cref{h^0 and h^1 of a shtuka} and \cref{cohomology of irreducible shtukas}) for shtukas over a field with zeros and poles satisfying condition (\ref{zero and pole not related by Frobenius}). In particular, if such a shtuka is irreducible, the dimension of its cohomology behaves as if the shtuka corresponds to a Drinfeld module. (See the remark following Proposition 3 in~\cite{Dr77} for the case of Drinfeld modules.)

In Section 4, we study the deformation theory of shtukas as a preparation for the next two sections. With the notion of toy shtukas introduced in~\cite{Ding18}, deformation theory is no longer necessary for our main result. But we  hope that it is interesting of its own right.

Sections 5 and 6 are devoted to the proof of our main result \cref{divisor of determinant of cohomology of shtukas}, which describes the determinant of cohomology of the universal shtuka over the moduli stack of shtukas. In Section 5, we use deformation theory to calculate the cohomology of shtukas over formal neighborhoods of a point on the moduli stack. In Section 6, we use the result in Section 5 to obtain the multiplicities of cohomology of shtukas at horospherical divisors.

In Section 7, we construct some rational functions on the moduli scheme of shtukas  with all level structures. The existence of these rational functions follows from \cref{cohomology of irreducible shtukas}. With \cref{divisor of determinant of cohomology of shtukas}, we are able to determine the divisors of these rational functions, and therefore explicitly exhibit some principal horospherical divisors of the moduli scheme.

\subsection{Acknowledgements}
The research was partially supported by NSF grant DMS-1303100. I would like to express my deep gratitude to my advisor Vladimir Drinfeld for his continual guidance and support. I thank Alexander Beilinson and Madhav Nori for helpful discussions.

\section{Review of Drinfeld shtukas}
\subsection{Definition of Drinfeld shtukas}
Let $S$ be a scheme over $\mathbb{F}_q$. Denote $\Phi=\Id_X\times\Fr_S:X\times S\to X\times S$.
\begin{Definition}[Drinfeld]
  A \emph{left shtuka} of rank $d$ over $S$ is a diagram
  \[\Phi^*\mathscr{F}\xhookleftarrow{j}\mathscr{F}'\xhookrightarrow{i}\mathscr{F}\]
  of locally free sheaves of rank $d$ on $X\times S$, where $i$ and $j$ are  injective morphisms, the cokernel of $i$ is an invertible sheaf on the graph $\Gamma_\alpha$ of some morphism $\alpha:S\to X$, the cokernel of $j$ is an invertible sheaf on the graph $\Gamma_\beta$ of some morphism $\beta:S\to X$.

  A \emph{right shtuka} of rank $d$ over $S$ is a diagram
  \[\Phi^*\mathscr{F}\xhookrightarrow{f}\mathscr{F}'\xhookleftarrow{g}\mathscr{F}\]
  of locally free sheaves of rank $d$ on $X\times S$, where $f$ and $g$ are  injective morphisms, the cokernel of $f$ is an invertible sheaf on the graph $\Gamma_\alpha$ of some morphism $\alpha:S\to X$, the cokernel of $g$ is an invertible sheaf on the graph $\Gamma_\beta$ of some morphism $\beta:S\to X$.

  We say that $\alpha$ is the zero of the shtuka and $\beta$ is the pole of the shtuka.
\end{Definition}

\begin{Definition}[Drinfeld]
For two morphisms $\alpha,\beta:S\to X$, a \emph{shtuka} of rank $d$ over $S$ with zero $\alpha$ and pole $\beta$ is a locally free sheaf $\mathscr{F}$ of rank $d$ on $X\times S$ equipped with an injective morphism $\Phi^*\mathscr{F}\to\mathscr{F}(\Gamma_\beta)$ inducing an isomorphism $\Phi^*\det\mathscr{F}\xrightarrow{\sim}(\det\mathscr{F})(\Gamma_\beta-\Gamma_\alpha)$, such that the image of the composition $\Phi^*\mathscr{F}\to\mathscr{F}(\Gamma_\beta)\to\mathscr{F}(\Gamma_\beta)/\mathscr{F}$ has rank at most 1 at $\Gamma_\beta$.
\end{Definition}

\begin{Remark}
  When the zero and the pole have disjoint graph, Construction A in \cref{(Section)general constructions for shtukas} shows that there is no difference between a shtuka, a left shtuka and a right shtuka. This remains true after applying any powers of partial Frobeniuses (see \cref{Section: partial Frobenius}) if we impose condition (\ref{zero and pole not related by Frobenius}) on the zero and the pole.

  From now on, when the zero and the pole satisfy condition (\ref{zero and pole not related by Frobenius}), we do not distinguish left and right shtukas, and simply call them shtukas.
\end{Remark}

\begin{Lemma}
  Let $\mathscr{M}$ be a quasi-coherent sheaf on $X$ and let $\widetilde{\mathscr{M}}$ be the pullback of $\mathscr{M}$ under the projection $X\times S\to X$. Then we have a canonical isomorphism $\Phi^*\widetilde{\mathscr{M}}\cong\widetilde{\mathscr{M}}$. \qed
\end{Lemma}

\begin{Definition}[Drinfeld]
Let $\mathscr{F}$ be a shtuka (resp. a left shtuka, resp. a right shtuka) of rank $d$ over $S$ with zero $\alpha$ and pole $\beta$. Let $D$ be a finite subscheme of $X$ such that $\alpha$ and $\beta$ map to $X-D$. A structure of level $D$ on $\mathscr{F}$ is an isomorphism $\iota:\mathscr{F}\otimes\mathscr{O}_{D\times S}\xrightarrow{\sim}\mathscr{O}_{D\times S}^d$ such that the following diagram commutes:
\[\begin{tikzcd}
  \mathscr{F}\otimes\mathscr{O}_{D\times S}\arrow[r,"\iota","\sim"']&\mathscr{O}_{D\times S}^d\\
  \Phi^*\mathscr{F}\otimes\mathscr{O}_{D\times S}\arrow[r,"\Phi^*\iota","\sim"']\arrow[u,"\sim"]&\Phi^*\mathscr{O}_{D\times S}^d\arrow[u,"\sim"]
\end{tikzcd}\]
\end{Definition}

\subsection{General constructions for shtukas}\label{(Section)general constructions for shtukas}
We have the following constructions for shtukas, which induce morphisms between moduli stacks of shtukas.

\textbf{Construction A:}

(i)  Let $\Phi^*\mathscr{G}\xhookleftarrow{j}\mathscr{F}\xhookrightarrow{i}\mathscr{G}$ be a left shtuka with zero $\alpha$ and pole $\beta$ such that $\Gamma_\alpha\cap\Gamma_\beta=\emptyset$. We form the pushout diagram
\[\begin{tikzcd}
    \mathscr{G}\arrow{r}{g}&\mathscr{H}\\
    \mathscr{F}\arrow{r}{j}\arrow{u}{i}&\Phi^*\mathscr{G}\arrow{u}{f}
\end{tikzcd}\]
Then $\Phi^*\mathscr{G}\xhookrightarrow{f}\mathscr{H}\xhookleftarrow{g}\mathscr{G}$ is a right shtuka of the same rank, with the same zero and pole.

(ii)  Let $\Phi^*\mathscr{G}\xhookrightarrow{f}\mathscr{F}\xhookleftarrow{g}\mathscr{G}$ be a right shtuka with zero $\alpha$ and pole $\beta$ such that $\Gamma_\alpha\cap\Gamma_\beta=\emptyset$. We form the pullback diagram
\[\begin{tikzcd}
    \mathscr{G}\arrow{r}{g}&\mathscr{F}\\
    \mathscr{H}\arrow{r}{j}\arrow{u}{i}&\Phi^*\mathscr{G}\arrow{u}{f}
\end{tikzcd}\]
Then $\Phi^*\mathscr{G}\xhookleftarrow{j}\mathscr{H}\xhookrightarrow{i}\mathscr{G}$ is a left shtuka of the same rank, with the same zero and pole.

\textbf{Construction B:}

(i) For a left shtuka $\Phi^*\mathscr{G}\xhookleftarrow{j}\mathscr{F}\xhookrightarrow{i}\mathscr{G}$ with zero $\alpha$ and pole $\beta$, we can construct a right shtuka $\Phi^*\mathscr{F}\xhookrightarrow{\Phi^*i}\Phi^*\mathscr{G}\xhookleftarrow{j}\mathscr{F}$ of the same rank, with zero $\Fr_X\bcirc\alpha$ and pole $\beta$.

(ii) For a right shtuka $\Phi^*\mathscr{F}\xhookrightarrow{f}\mathscr{G}\xhookleftarrow{g}\mathscr{F}$ with zero $\alpha$ and pole $\beta$, we can construct a left shtuka $\Phi^*\mathscr{G}\xhookleftarrow{\Phi^*g}\Phi^*\mathscr{F}\xhookrightarrow{f}\mathscr{G}$ of the same rank, with zero $\alpha$ and pole $\Fr_X\bcirc\beta$.

\textbf{Construction C:}
For a shtuka $\mathscr{F}$ with zero $\alpha$ and pole $\beta$ satisfying $\Gamma_\alpha\cap\Gamma_\beta=\emptyset$, its dual $\mathscr{F}^\vee$ is a shtuka of the same rank, with zero $\beta$ and pole $\alpha$.

\textbf{Construction D:}
Let $\mathscr{F}$ be a shtuka over $S$. Let $\mathscr{L}$ be an invertible sheaf on $X$, and let $\widetilde{\mathscr{L}}$ be the pullback of $\mathscr{L}$ under the projection $X\times S\to X$. then $\mathscr{F}\otimes\widetilde{\mathscr{L}}$ is a shtuka of the same rank, with the same zero and pole.

\textbf{Construction D':}
Suppose in Construction D the shtuka $\mathscr{F}$ is equipped with a structure of level $D$ and the invertible sheaf $\mathscr{L}$ is trivialized at $D$. Then the shtuka $\mathscr{F}\otimes\widetilde{\mathscr{L}}$ is naturally equipped with a structure of level $D$.

\textbf{Construction E:}
Let $\mathscr{F}$ be a shtuka of rank $d$ equipped with a structure of level $D$. Suppose we are given an $\mathscr{O}_D$-submodule $\mathscr{R}\subset \mathscr{O}_D^d$. Let $\mathscr{F}'$ be the kernel of the composition $\mathscr{F}\to\mathscr{F}\otimes_{\mathscr{O}_{X\times S}}\mathscr{O}_{D\times S}\xrightarrow{\sim}\mathscr{O}_{D\times S}^d\to\mathscr{O}_{D\times S}^d/(\mathscr{R}\boxtimes\mathscr{O}_S)$. Then $\mathscr{F}'$ is a shtuka of the same rank, with the same zero and pole.

\textbf{Construction E':}
Suppose in Construction E we are given a surjective morphism of $\mathscr{O}_D$-modules $\mathscr{R}\to\mathscr{O}_{D'}^d$, where $D'$ is a subscheme of $D$. Then the composition $\mathscr{F}'\to\mathscr{R}\boxtimes\mathscr{O}_S\to\mathscr{O}_{D'\times S}^d$ defines a structure of level $D'$ on $\mathscr{F}'$.

\subsection{Partial Frobenius}\label{Section: partial Frobenius}
\begin{Definition}
Let $F_1$ be the construction of first applying A(ii) and then applying B(i). Let $F_2$ be the construction of first applying B(ii) and then applying A(i). They are called partial Frobenius.
\end{Definition}

\begin{Proposition}
For a shtuka $\mathscr{F}$ over $S$ with zero and pole satisfying condition (\ref{zero and pole not related by Frobenius}), we have natural isomorphisms $F_1F_2\mathscr{F}\cong F_2F_1\mathscr{F}\cong\Fr_S^*\mathscr{F}$. \qed
\end{Proposition}

\begin{Remark}
  If a shtuka $\mathscr{F}$ over $S$ has zero $\alpha$ and pole $\beta$, then $F_1\mathscr{F}$ has zero $\alpha\bcirc\Fr_S=\Fr_X\bcirc\alpha$ and pole $\beta$, and $F_2\mathscr{F}$ has zero $\alpha$ and pole $\beta\bcirc\Fr_S=\Fr_X\bcirc\beta$.
\end{Remark}

\begin{Lemma}\label{inverse of partial Frobenius}
  Let $E$ be a perfect field over $\mathbb{F}_q$. Let $\mathscr{F}$ be a shtuka over $\Spec E$ with zero $\alpha$ and pole $\beta$ satisfying condition (\ref{zero and pole not related by Frobenius}). Then there is a unique subsheaf $\mathscr{H}\subset\mathscr{F}$ satisfying the following conditions.

  (i) The image of $\Phi^*\mathscr{H}$ in $\mathscr{F}(\beta)$ is contained in $\mathscr{F}\cap\mathscr{H}(\beta\bcirc\Fr_E^{-1})$.

  (ii) The sheaf $\mathscr{H}$ equipped with the morphism $\Phi^*\mathscr{H}\to\mathscr{H}(\beta\bcirc\Fr_E^{-1})$ is a shtuka over $\Spec E$ with zero $\alpha$ and pole $\beta\bcirc\Fr_E^{-1}$.

  (iii) Applying partial Frobenius $F_2$ to $\mathscr{H}$, we get the original shtuka $\mathscr{F}$.
\end{Lemma}
\begin{proof}
  A subsheaf $\mathscr{H}\subset\mathscr{F}$ satisfies condition (i), (ii) and (iii) if and only if the image of $\Phi^*\mathscr{H}$ in $\mathscr{F}(\beta)$ equals $\Phi^*\mathscr{F}\cap\mathscr{F}$. Since $E$ is perfect, such $\mathscr{H}$ uniquely exists.
\end{proof}

\begin{Remark}\label{all integer powers of partial Frobenius}
For a shtuka $\mathscr{F}$ over a perfect field with zero and pole satisfying condition (\ref{zero and pole not related by Frobenius}), \cref{inverse of partial Frobenius} allows us to define its partial Frobeniuses $(F_2)^i\mathscr{F}$ for all $i\in \mathbb{Z}$. We have $\chi((F_2)^i\mathscr{F})=\chi(\mathscr{F})+i$ for all $i\in \mathbb{Z}$.
\end{Remark}

\section{Reducible shtukas over a field}
In this section, we fix a field $E$ over $\mathbb{F}_q$.

\subsection{Definitions}
\begin{Definition}\label{definition of reducible shtukas}
A shtuka $\mathscr{F}$ over $\Spec E$ of rank $d$ with zero $\alpha$ and pole $\beta$ is said to be \emph{reducible} if $\mathscr{F}$ contains a nonzero subsheaf $\mathscr{E}$ of rank $<d$ such that the image of $\Phi^*\mathscr{E}$ in $\mathscr{F}(\beta)$ is contained in $\mathscr{E}(\beta)$.
\end{Definition}

\begin{Lemma}\label{saturation of a Frobenius stable subsheaf}
  With the same notation as in \cref{definition of reducible shtukas}, if $\mathscr{G}$ is the saturation of $\mathscr{E}$ in $\mathscr{F}$, then the image of $\Phi^*\mathscr{G}$ in $\mathscr{F}(\beta)$ is contained in $\mathscr{G}(\beta)$, $(\Phi^*\mathscr{G}+\mathscr{G})/\mathscr{G}$ is a torsion sheaf at $\beta$ of length at most one, and $(\Phi^*\mathscr{G}+\mathscr{G})/\Phi^*\mathscr{G}$ is a torsion sheaf at $\alpha$ of length at most one.
\end{Lemma}
\begin{proof}
  Let $\eta=\Spec E'$ be the generic point of $X\otimes E$. Recall that $k$ is the field of rational functions on $X$. The injective homomorphism $\Id_k\otimes\Fr_E:k\otimes E\to k\otimes E$ induces a homomorphism $\psi: E'\to E'$. The morphism $\Phi^*\mathscr{F}\hookrightarrow\mathscr{F}(\beta)$ induces a linear map of $E'$-vectors spaces $\phi:\psi^*\mathscr{F}_\eta\to\mathscr{F}_\eta$. The condition $\Phi^*\mathscr{E}\subset\mathscr{E}(\beta)$ implies that $\phi(\psi^*\mathscr{E}_\eta)=\mathscr{E}_\eta$. Note that $\mathscr{F}_\eta=(\Phi^*\mathscr{F})_\eta$. Thus the two subbundles $\mathscr{G}\subset\mathscr{F}$ and $\Phi^*\mathscr{G}\subset\Phi^*\mathscr{F}$ correspond to the same $E'$-subspace $\mathscr{E}_\eta\subset\mathscr{F}_\eta$. Since $\mathscr{G}(\beta)$ is saturated in $\mathscr{F}(\beta)$, we have $\Phi^*\mathscr{G}\subset\mathscr{G}(\beta)$. The definition of a shtuka implies that $(\Phi^*\mathscr{F}+\mathscr{F})/\mathscr{F}$ is a torsion sheaf at $\beta$ of length at most one, hence is $(\Phi^*\mathscr{G}+\mathscr{G})/\mathscr{G}$. The proof for $(\Phi^*\mathscr{G}+\mathscr{G})/\Phi^*\mathscr{G}$ is similar.
\end{proof}

\begin{Remark}\label{two types of exact sequences for reducible shtukas}
  Let the notation be the same as in \cref{definition of reducible shtukas}. Let $\mathscr{G}$ be the saturation of $\mathscr{E}$ in $\mathscr{F}$. We have an exact sequence of locally free sheaves on $X\otimes E$
\[\begin{tikzcd}0\arrow[r]&\mathscr{G}\arrow[r]&\mathscr{F}\arrow[r]&\mathscr{H}\arrow[r]&0.\end{tikzcd}\]

Assume $\alpha\ne \beta$. From \cref{saturation of a Frobenius stable subsheaf} we see that one (and only one) of the following two possibilities holds.

(1) $\mathscr{G}$ is a shtuka with zero $\alpha$ and pole $\beta$, and the morphism $\Phi^*\mathscr{F}\hookrightarrow\mathscr{F}(\beta)$ induces an isomorphism $\Phi^*\mathscr{H}\xrightarrow{\sim}\mathscr{H}$.

(2) $\mathscr{H}$ is a shtuka with zero $\alpha$ and pole $\beta$, and the morphism $\Phi^*\mathscr{F}\hookrightarrow\mathscr{F}(\beta)$ induces an isomorphism $\Phi^*\mathscr{G}\xrightarrow{\sim}\mathscr{G}$.
\end{Remark}

\begin{Remark}\label{Remark: exact sequence for reducible shtuka}
Let $\mathscr{F}$ be a shtuka over $\Spec E$ with zero and pole satisfying condition (\ref{zero and pole not related by Frobenius}). Let $n\in\mathbb{Z}$ if $E$ is perfect and $n\in\mathbb{Z}_{\ge0}$ in general. Let $\mathscr{F}_n$ be the shtuka obtained from $\mathscr{F}$ by applying partial Frobenius $(F_2)^n$. (See \cref{all integer powers of partial Frobenius}.)

Assume $\mathscr{F}$ is reducible, and let $\mathscr{G},\mathscr{H}$ be as in \cref{two types of exact sequences for reducible shtukas}. Then $\mathscr{F}_n$ is also irreducible with the following more precise description.

In case (1) of \cref{two types of exact sequences for reducible shtukas}, we have an exact sequence of locally free sheaves on $X\otimes E$
\[\begin{tikzcd}0\arrow[r]&\mathscr{G}_n\arrow[r]&\mathscr{F}_n\arrow[r]&\mathscr{H}\arrow[r]&0\end{tikzcd}\]
in which $\mathscr{G}_n$ is the shtuka obtained from $\mathscr{G}$ by applying $(F_2)^n$, and $\Phi^*\mathscr{F}_n\hookrightarrow\mathscr{F}_{n+1}$ induces an isomorphism $\Phi^*\mathscr{H}\xrightarrow{\sim}\mathscr{H}$.

In case (2) of \cref{two types of exact sequences for reducible shtukas}, we have an exact sequence of locally free sheaves on $X\otimes E$.
\[\begin{tikzcd}0\arrow[r]&\mathscr{G}\arrow[r]&\mathscr{F}_n\arrow[r]&\mathscr{H}_n\arrow[r]&0\end{tikzcd}\]
in which $\mathscr{H}_n$ is the shtuka obtained from $\mathscr{H}$ by applying $(F_2)^n$, and $\Phi^*\mathscr{F}_n\hookrightarrow\mathscr{F}_{n+1}$ induces an isomorphism $\Phi^*\mathscr{G}\xrightarrow{\sim}\mathscr{G}$.
\end{Remark}

\subsection{Maximal trivial sub and maximal trivial quotient}\label{section: maximal trivial sub and maximal trivial quotient}
\begin{Lemma}\label{difference between divisor and its Frobenius}
Let $\alpha,\beta:\Spec E\to X$ be two morphisms. If an effective divisor $D$ of $X\otimes E$ satisfies $\Phi^*D+\alpha=D+\beta$, then $\beta=\Fr_X^n\bcirc\alpha$ for some $n\ge 0$.
\end{Lemma}
\begin{proof}
We prove by induction on the degree of $D$. When $D=0$, we get $\alpha=\beta$. When $D>0$ and $\alpha\ne \beta$, set $D'=D-\alpha$. Then $D'$ is effective,  $\deg D'<\deg D$, and $\Phi^*D'-D'=\beta-\Fr_X\bcirc\alpha$. The induction hypothesis shows that $\beta=\Fr_X^m\bcirc\Fr_X\bcirc\alpha=\Fr_X^{m+1}\bcirc\alpha$ for some $m\ge 0$.
\end{proof}

\begin{Proposition}\label{Frobenius stable subsheaf (equal rank)}
Let $\mathscr{G}$ be a shtuka over $\Spec E$ with zero $\alpha $ and pole $\beta$. Assume that there exists a subsheaf $\mathscr{E}\subset\mathscr{G}$ satisfying

(i) $\rank\mathscr{E}=\rank\mathscr{G}$;

(ii) the image of $\Phi^*\mathscr{E}$ in $\mathscr{G}(\beta)$ is contained in $\mathscr{E}$.

\noindent Then $\beta=\Fr_X^n\bcirc\alpha$ for some $n\ge 0$.
\end{Proposition}
\begin{proof}
The $d$-th exterior power of a shtuka of rank $d$ is a shtuka of rank 1 with the same zero and pole. So we can apply the $d$-th exterior power to all sheaves involved to reduce the problem to the case $d=1$.

Now we have $\mathscr{G}=\mathscr{E}(D)$, where $D$ is an effective divisor of $X\otimes E$. We see that $\Phi^*\mathscr{G}=(\Phi^*\mathscr{E})(\Phi^*D).$ Since $\chi(\Phi^*\mathscr{E})=\chi(\mathscr{E})$, the assumption $f(\Phi^*\mathscr{E})\subset g(\mathscr{E})$ implies $f(\Phi^*\mathscr{E})=g(\mathscr{E})$. Hence $\Phi^*\mathscr{G}\cong\mathscr{E}(\Phi^*D)$. Since the shtuka has rank 1, we have $\Phi^*\mathscr{G}\cong\mathscr{G}(\beta-\alpha)=\mathscr{E}(D+\beta-\alpha)$. Thus we get $\Phi^*D=D+\beta-\alpha$. The statement now follows from \cref{difference between divisor and its Frobenius}.
\end{proof}

\begin{Proposition}\label{maximal trivial sub and maximal trivial quotient}
Let $\mathscr{F}$ be a shtuka over $\Spec E$ with zero $\alpha$ and pole $\beta$ satisfying condition (\ref{zero and pole not related by Frobenius}). Let $S_1$ (resp. $S_2$) be the poset of all subsheaves $\mathscr{E}\subset\mathscr{F}$ satisfying the following condition (1) (resp. (2)).

(1) The image of $\Phi^*\mathscr{E}$ in $\mathscr{F}(\beta)$ is contained in $\mathscr{E}(\beta)$, the sheaf $\mathscr{F}/\mathscr{E}$ is locally free, and the morphism $\Phi^*(\mathscr{F}/\mathscr{E})\to(\mathscr{F}/\mathscr{E})(\beta)$ induces an isomorphism $\Phi^*(\mathscr{F}/\mathscr{E})\xrightarrow{\sim}(\mathscr{F}/\mathscr{E})$.

(2) The image of $\Phi^*\mathscr{E}$ in $\mathscr{F}(\beta)$ is $\mathscr{E}$.

Then the poset $S_1$ has a least element, denoted by $\mathscr{F}^{\I}$. The poset $S_2$ has a greatest element, denoted by $\mathscr{F}^{\II}$, and $\mathscr{F}/\mathscr{F}^{\II}$ is locally free.
\end{Proposition}

\begin{proof}
The set $S_2$ is nonempty since it contains the zero subsheaf. Let $\mathscr{F}^{\II}=\sum_{\mathscr{E}\in S_2}\mathscr{E}$. We see that the image of $\Phi^*\mathscr{F}^{\II}$ in $\mathscr{F}(\beta)$ is contained in $\mathscr{F}^{\II}$. Since $\chi(\Phi^*\mathscr{F}^{\II})=\chi(\mathscr{F}^{\II})$, the image of $\Phi^*\mathscr{F}^{\II}$ in $\mathscr{F}(\beta)$ is equal to $\mathscr{F}^{\II}$. Thus $\mathscr{F}^{\II}$ is the greatest element of $S_2$.

Suppose that $\mathscr{F}/\mathscr{F}^{\II}$ is not locally free. Let $\mathscr{G}$ be the saturation of $\mathscr{F}^{\II}$ in $\mathscr{F}$. We have $\mathscr{F}^{\II}\subsetneqq\mathscr{G}$. Hence the image of $\Phi^*\mathscr{G}$ in $\mathscr{F}(\beta)$ is not contained in $\mathscr{G}$. Now \cref{two types of exact sequences for reducible shtukas} implies that $\mathscr{G}$ is a shtuka with zero $\alpha$ and pole $\beta$. Applying \cref{Frobenius stable subsheaf (equal rank)} to $\mathscr{G}$ and its subsheaf $\mathscr{F}^{\II}$, we deduce that $\beta=\Fr_X^n\bcirc\alpha$ for some $n\ge 0$, contradicting condition (\ref{zero and pole not related by Frobenius}).

Let $\mathscr{F}^\vee$ be the dual of the given shtuka $\mathscr{F}$. It has zero $\beta$ and pole $\alpha$ which also satisfy condition (\ref{zero and pole not related by Frobenius}). Therefore, there is a unique maximal subsheaf $(\mathscr{F}^\vee)^{\II}\subset\mathscr{F}^\vee$ satisfying condition (2), and $\mathscr{F}^\vee/(\mathscr{F}^\vee)^{\II}$ is locally free. Then $\mathscr{F}^{\I}=(\mathscr{F}^\vee/(\mathscr{F}^\vee)^{\II})^\vee$ is a least element of $S_1$.
\end{proof}

\begin{Remark}\label{criterion for irreducibility of shtuka using maximal trivial sub and quotient}
Suppose $\mathscr{F}$ is a right shtuka over $\Spec E$ with zero and pole satisfying condition (\ref{zero and pole not related by Frobenius}). Then $\mathscr{F}$ is irreducible if and only if $\mathscr{F}^{\I}=\mathscr{F}$ and $\mathscr{F}^{\II}=0$.
\end{Remark}

\section{Cohomology of Shtukas over a Perfect Field}
In this section, we fix a perfect field $E$ over $\mathbb{F}_q$. Denote $\Phi=\Id_X\otimes\Fr_E:X\otimes E\to X\otimes E$.

\subsection{Notation}\label{notation for cohomology of shtukas over a perfect field}
Let $\mathscr{F}_0$ be a shtuka over $\Spec E$ with zero $\alpha$ and pole $\beta$ satisfying condition (\ref{zero and pole not related by Frobenius}). We further assume that $\chi(\mathscr{F}_0)=0$. As in \cref{all integer powers of partial Frobenius}, for each $i\in\mathbb{Z}$ we obtain a shtuka $\mathscr{F}_i$ from the above one by applying partial Frobenius $(F_2)^i$. We get a chain of inclusions
\[\dots\hookrightarrow\mathscr{F}_i\hookrightarrow\mathscr{F}_{i+1}\hookrightarrow\mathscr{F}_{i+2}\hookrightarrow\dots\]
We have $\chi(\mathscr{F}_i)=i$ for all $i\in\mathbb{Z}$.

For a right shtuka $\mathscr{F}$ over $\Spec E$, let $\mathscr{F}^{\I},\mathscr{F}^{\II}$ be as in \cref{maximal trivial sub and maximal trivial quotient}.

\cref{Remark: exact sequence for reducible shtuka} allows us to identify $\mathscr{F}_i^{\II}$ for all $i\in \mathbb{Z}$, and we denote it by $\mathscr{A}$. Similarly we can identify $\mathscr{F}_i/\mathscr{F}_i^{\I}$ for all $i\in \mathbb{Z}$, and we denote it by $\mathscr{B}$. We know that $\mathscr{A}$ is a subbundle of all $\mathscr{F}_i$, and $\mathscr{B}$ is a quotient bundle of all $\mathscr{F}_i$. Put $a=h^0(X\otimes E,\mathscr{A}),b=h^1(X\otimes E,\mathscr{B})$.

\subsection{Computation of cohomology}
Recall that $\chi(\mathscr{F}_i)=i$ for all $i\in \mathbb{Z}$.

The goal of this section is to prove the following proposition.
\begin{Proposition}\label{cohomology of shtukas over a field}
The morphism $\mathscr{A}\to \mathscr{F}_i$ induces an isomorphism $H^0(X\otimes E,\mathscr{A})\xrightarrow{\sim}H^0(X\otimes E,\mathscr{F}_i)$ for $i\le a-b$. The morphism $\mathscr{F}_i\to\mathscr{B}$ induces an isomorphism $H^1(X\otimes E,\mathscr{F}_i)\xrightarrow{\sim}H^1(X\otimes E,\mathscr{B})$ for $i\ge a-b$.
\end{Proposition}

\begin{Corollary}\label{h^0 and h^1 of a shtuka}
Let $\mathscr{G}$ be a shtuka over $\Spec E$ with zero and pole satisfying condition (\ref{zero and pole not related by Frobenius}). Put $m=h^0(X\otimes E,\mathscr{G}^{\II}), n=h^1(X\otimes E,\mathscr{G}/\mathscr{G}^{\I})$. Then $h^0(\mathscr{G})=\max\{m,n+\chi(\mathscr{G})\}$ and $h^1(\mathscr{G})=\max\{n,m-\chi(\mathscr{G})\}$. \qed
\end{Corollary}

The following corollary follows from \cref{h^0 and h^1 of a shtuka,criterion for irreducibility of shtuka using maximal trivial sub and quotient}.
\begin{Corollary}\label{cohomology of irreducible shtukas}
If $\mathscr{G}$ is an irreducible shtuka over $\Spec E$ with zero and pole satisfying condition (\ref{zero and pole not related by Frobenius}) and we have $\chi(\mathscr{G})=0$, then $H^0(X\otimes E,\mathscr{G})=0$ and $H^1(X\otimes E, \mathscr{G})=0$. \qed
\end{Corollary}

For $i\in \mathbb{Z}$, put $M_i=H^0(X\otimes E,\mathscr{F}_i)$. We have inclusions $\Phi_E^*M_i\subset M_{i+1}$ induced by inclusions $\Phi^*\mathscr{F}_i\hookrightarrow\mathscr{F}_{i+1}$. Put $M_A=H^0(X\otimes E,\mathscr{A})$.

\begin{Lemma}\label{increments in a sequence of shtukas are isomorphic under Frobenius}
  For all $i\in\mathbb{Z}$, the composition $\Phi^*(\mathscr{F}_i/\mathscr{F}_{i-1})\xrightarrow{\sim}\Phi^*\mathscr{F}_i/\Phi^*\mathscr{F}_{i-1} \to\mathscr{F}_{i+1}/\mathscr{F}_i$ is an isomorphism.
\end{Lemma}
\begin{proof}
  Since $\alpha$ and $\beta$ satisfy condition (\ref{zero and pole not related by Frobenius}), we have $\Phi^*\mathscr{F}_{i-1}=\Phi^*\mathscr{F}_i\cap\mathscr{F}_i$ for all $i\in\mathbb{Z}$. Hence the composition is injective. Since both $\Phi^*(\mathscr{F}_i/\mathscr{F}_{i-1})$ and $\mathscr{F}_{i+1}/\mathscr{F}_i$ have length one, the statement follows.
\end{proof}

\begin{Lemma}\label{increase of H^0 after the first one}
If $M_{i-1}\subsetneqq M_i$, then $M_{j-1}\subsetneqq M_j$ for all $j>i$.
\end{Lemma}
\begin{proof}
  Pick $s\in M_i$ such that $s\notin M_{i-1}$. Then the image of $s$ in $H^0(X\otimes E,\mathscr{F}_i/\mathscr{F}_{i-1})$ is nonzero. Hence the image of $\Phi^*s$ under the composition $\Phi^*M_i\hookrightarrow M_{i+1}\to H^0(X\otimes E,\mathscr{F}_{i+1}/\mathscr{F}_i)$ is nonzero by \cref{increments in a sequence of shtukas are isomorphic under Frobenius}. This implies that $\Phi^*s\notin M_i$. The statement follows by induction.
\end{proof}

\begin{Lemma}\label{asymptotic behavior for H^0(F_i) when i<<0}
We have $M_i=M_A$ when $i$ is sufficiently small.
\end{Lemma}
\begin{proof}
Put $M_{-\infty}=\bigcap_{i\in \mathbb{Z}}M_i$. We have $\Phi_E^*M_{-\infty}\subset M_{-\infty}$, and hence $\Phi^*_E M_{-\infty}=M_{-\infty}$ by dimension comparison. Since each $M_i$ is finite dimensional, there exists $j\in\mathbb{Z}$ such that $M_j=M_{-\infty}$. Thus $\Phi_E^* M_j=M_j$.  Let $\mathscr{E}$ be the $\mathscr{O}_{X\otimes E}$-submodule of $\mathscr{F}_j$ generated by its global sections. Then $\Phi^*\mathscr{F}_j\hookrightarrow\mathscr{F}_{j+1}$ induces an isomorphism $\Phi^*\mathscr{E}\xrightarrow{\sim}\mathscr{E}$. By the definition of $\mathscr{A}$ we have $\mathscr{E}\subset\mathscr{A}$ . Hence $M_j\subset M_A$. On the other hand, we have $\mathscr{A}\subset\mathscr{F}_j$. Therefore, $M_A=M_j=M_{-\infty}$, and we see that $M_i=M_A$ for all $i\le j$.
\end{proof}

\begin{Lemma}\label{asymptotic behavior for H^1(F_i) when i>>0}
The morphism $\mathscr{F}_i\to\mathscr{B}$ induces an isomorphism $H^1(X\otimes E, \mathscr{F}_i)\xrightarrow{\sim}H^1(X\otimes E,\mathscr{B})$ when $i$ is sufficiently large.
\end{Lemma}
\begin{proof}
The statement follows from \cref{asymptotic behavior for H^0(F_i) when i<<0}, Serre duality and base change for $H^1$ under $\Phi^*$.
\end{proof}

\begin{proof}[Proof of \cref{cohomology of shtukas over a field}]
  Since $\mathscr{F}_{i+1}/\mathscr{F}_i$ has length one for all $i\in \mathbb{Z}$, we know that $h^0(X\otimes E,\mathscr{F}_{i+1})-h^0(X\otimes E,\mathscr{F}_i)$ equals 0 or 1. Recall that $\chi(\mathscr{F}_i)=i$ for all $i\in \mathbb{Z}$. Now the statement follows from \cref{increase of H^0 after the first one,asymptotic behavior for H^0(F_i) when i<<0,asymptotic behavior for H^1(F_i) when i>>0}.
\end{proof}

\section{Deformation of Shtukas}
In this section, we fix a field $E$ over $\mathbb{F}_q$. For a scheme $S$ over $\Spec E$, we denote $\mathsf{p}_S:S\to\Spec E$ to be the structure morphism.

\subsection{The category $\ArtAlg_E$}\label{Section: the category of Artinian E-algebras for deformation}
We define a category $\ArtAlg_E$. Its objects are local Artinian $E$-algebras whose residue fields are identified with $E$. We further require that for any $A\in\ArtAlg_E$, the composition $E\to A\to E$ is the identity. Morphisms in $\ArtAlg_E$ are local homomorphisms of $E$-algebra which induce the identity on the residue field.

Define $\ArtAlg_E^{(n)}$ to be the full subcategory of $\ArtAlg_E$ whose objects satisfy the condition that $x^{q^n}=0$ for any $x$ in the maximal ideal.

For any $A\in \ArtAlg_E$, let $\mathsf{p}_A:\Spec A\to \Spec E$ be the structure morphism, and let $e_A:\Spec E\to \Spec A$ be the closed point.

\subsection{Upper and lower modifications of vector bundles}
Let $S$ be a scheme over $\mathbb{F}_q$ equipped with a morphism $\lambda: S\to X$. Let $\mathscr{E}$ be a locally free sheaf of finite rank on $X\times S$. By a \emph{upper modification} of $\mathscr{E}$ at $\lambda$ we mean a locally free sheaf $\mathscr{E}'$ on $X\times S$ together with an injective morphism $\mathscr{E}\hookrightarrow\mathscr{E}'$ such that $\mathscr{E}'/\mathscr{E}$ is an invertible sheaf on $\Gamma_{\lambda}$. By a \emph{lower modification} of $\mathscr{E}$ at $\lambda$ we mean a locally free sheaf $\mathscr{E}'$ together with an injective morphism $\mathscr{E}'\hookrightarrow\mathscr{E}$ such that $\mathscr{E}/\mathscr{E}'$ is an invertible sheaf on $\Gamma_{\lambda}$.

\begin{Remark}\label{modifications at disjoint closed subschemes}
  Let $\lambda_1,\dots,\lambda_m$ be morphisms from $S$ to $X$. Suppose we are given upper/lower modifications $\mathscr{E}_i$ of a locally free sheaf $\mathscr{E}$ at each $\lambda_i$. If $\Gamma_{\lambda_i}(1\le i\le m)$ are mutually disjoint, we can form a locally free sheaf by applying modifications consecutively at each $\lambda_i$ according to $\mathscr{E}_i$, and the sheaf we obtain is independent of the order of modifications.
\end{Remark}

\begin{Remark}\label{bijection between modifications and morphisms to projective spaces}
  Let $T$ be a scheme over $\Spec E$ with structure morphism $\mathsf{p}_T:T\to\Spec E$. Let $\mathscr{E}$ be a locally free sheaf of finite rank on $X\times T$. For a morphism $\mu:\Spec E\to X$, there is a natural bijection between $E$-morphisms from $T$ to $\mathbb{P}(\mathscr{E}(\mu)/\mathscr{E})$ (resp. $\mathbb{P}(\mathscr{E}/\mathscr{E}(-\mu))$) and upper (resp. lower) modifications of $\mathscr{E}$ at $\mu\bcirc\mathsf{p}_T$.
\end{Remark}

\subsection{Notation for the shtuka to be deformed}
Fix a shtuka $\mathscr{F}$ over $\Spec E$ with zero $\alpha$ and pole $\beta$ satisfying condition (\ref{Frobenius shifts of zero and pole mutually disjoint}).

We put $\tensor[^{\tau^i}]{\mathscr{F}}{}=(\Phi_E^*)^i\mathscr{F}$.

We know that $\tensor[^{\tau^i}]{\mathscr{F}}{}$  is obtained from $\tensor[^{\tau^{i+1}}]{\mathscr{F}}{}$ by an upper modification at $\Fr_X^i\bcirc\alpha$ and a lower modification at $\Fr_X^i\bcirc\beta$. Note that for $i>0$, $\Gamma_{\Fr_X^{i}\bcirc\alpha}$ (resp. $\Gamma_{\Fr_X^{i}\bcirc\beta}$) is disjoint from $\Gamma_\alpha$ (resp. $\Gamma_\beta$). So we get natural isomorphisms
\begin{equation}\label{isomorphism of fiber at zero}
(\tensor[^\tau]{\mathscr{F}}{})(\alpha)/\tensor[^\tau]{\mathscr{F}}{}
\cong(\tensor[^{\tau^2}]{\mathscr{F}}{})(\alpha)/\tensor[^{\tau^2}]{\mathscr{F}}{}
\cong(\tensor[^{\tau^3}]{\mathscr{F}}{})(\alpha)/\tensor[^{\tau^3}]{\mathscr{F}}{}
\cong\dots
\end{equation}
\begin{equation}\label{isomorphism of fiber at pole}
\tensor[^\tau]{\mathscr{F}}{}/(\tensor[^\tau]{\mathscr{F}}{})(-\beta)
\cong\tensor[^{\tau^2}]{\mathscr{F}}{}/(\tensor[^{\tau^2}]{\mathscr{F}}{})(-\beta)
\cong\tensor[^{\tau^3}]{\mathscr{F}}{}/(\tensor[^{\tau^3}]{\mathscr{F}}{})(-\beta)
\cong\dots
\end{equation}
induced by upper and lower modifications.

We denote $P^\sharp=\mathbb{P}^\vee((\tensor[^\tau]{\mathscr{F}}{})(\alpha)/\tensor[^\tau]{\mathscr{F}}{})$, $P^\flat=\mathbb{P}(\tensor[^\tau]{\mathscr{F}}{}/(\tensor[^\tau]{\mathscr{F}}{})(-\beta))$, $P^\natural=P^\sharp\times P^\flat$.

For an $E$-scheme $S$, we denote $\mathsf{p}_S: S\to\Spec E$ to be the structure morphism. When $i\ge 1$, \cref{bijection between modifications and morphisms to projective spaces} says that there is a natural bijection between morphisms from $S$ to $P^\sharp$ (resp. $P^\flat$) and upper (resp. lower) modifications of $(\Id_X\times \mathsf{p}_S)^*(\tensor[^{\tau^i}]{\mathscr{F}}{})$ at $\alpha\bcirc\mathsf{p}_S$ (resp. $\beta\bcirc\mathsf{p}_S$).

We know that $\mathscr{F}$ is obtained from $\tensor[^\tau]{\mathscr{F}}{}$ by an upper modification at $\alpha$ and a lower modification at $\beta$. Let $p^\sharp:\Spec E\to P^\sharp$ (resp. $p^\flat:\Spec E\to P^\flat$) be the $E$-point corresponding to that upper (resp. lower) modification.

\subsection{Infinitesimal deformation of shtukas}
Let $A\in\ArtAlg_E$. A deformation of the given shtuka $\mathscr{F}$ to $\Spec A$ is a shtuka $\widetilde{\mathscr{F}}$ over $\Spec A$ together with an isomorphism $e_A^*\widetilde{\mathscr{F}}\xrightarrow{\sim}\mathscr{F}$. We say that two deformations $\widetilde{\mathscr{F}}_1$ and $\widetilde{\mathscr{F}}_2$ are isomorphic if there is an isomorphism $\widetilde{\mathscr{F}}_1\xrightarrow{\sim}\widetilde{\mathscr{F}}_2$ of shtukas compatible with the isomorphisms $e_A^*\widetilde{\mathscr{F}}_1\xrightarrow{\sim}\mathscr{F}$ and $e_A^*\widetilde{\mathscr{F}}_2\xrightarrow{\sim}\mathscr{F}$.

\subsection{Universal Deformation}\label{section: universal deformation}
Let $G_1:\ArtAlg_E\to \Sets$ be the functor which associates to each $A\in \ArtAlg_E$ the set of isomorphism classes of deformations of the given  shtuka $\mathscr{F}$ to $\Spec A$ with zero $\alpha\bcirc\mathsf{p}_A$ and pole $\beta\bcirc\mathsf{p}_A$. Let $G_2:\ArtAlg_E\to \Sets$ be the functor which associates to each $A\in \ArtAlg_E$ the set of $E$-morphisms $f:\Spec A\to P^\natural$ such that $f\bcirc e_A=p^\natural$. Let $G_1^{(n)},G_2^{(n)}$ be the restrictions of the functors $G_1,G_2$ to $\ArtAlg_E^{(n)}$.

Now we define a morphism of functors $\xi:G_1\to G_2$. Let $A$ be an object of $\ArtAlg_E^{(n)}$. Let $\widetilde{\mathscr{F}}$ be a deformation of the given shtuka to $\Spec A$ with zero $\alpha\bcirc\mathsf{p}_A$ and pole $\beta\bcirc\mathsf{p}_A$.  Then we have an isomorphism
\[(\Phi_A^*)^n\widetilde{\mathscr{F}}\cong (\Id_X\times \mathsf{p}_A)^*(\tensor[^{\tau^n}]{\mathscr{F}}{})\]
So $\widetilde{\mathscr{F}}$ is obtained from $(\Id_X\times \mathsf{p}_A)^*(\tensor[^{\tau^n}]{\mathscr{F}}{})$ by upper modifications at $\alpha\bcirc\mathsf{p}_A$, $\Fr_X\bcirc\alpha\bcirc\mathsf{p}_A$, $\dots$, $\Fr_X^{n-1}\bcirc\alpha\bcirc\mathsf{p}_A$ and lower modifications at $\beta\bcirc\mathsf{p}_A,\Fr_X\bcirc\beta\bcirc\mathsf{p}_A,\dots,\Fr_X^{n-1}\bcirc\beta\bcirc\mathsf{p}_A$. Note that the modifications are at disjoint closed subschemes of $X\otimes A$ since $\alpha$ and $\beta$ satisfy condition (\ref{Frobenius shifts of zero and pole mutually disjoint}), so \cref{modifications at disjoint closed subschemes} applies.

For the modifications at $\alpha\bcirc\mathsf{p}_A$ and $\beta\bcirc\mathsf{p}_A$, we get two morphisms $f^\sharp:\Spec A\to P^\sharp$ and $f^\flat:\Spec A\to P^\flat$. Put $f^\natural=f^\sharp\times f^\flat:\Spec A\to P^\natural$. We see that $f^\natural\bcirc e_A=p^\natural$. If two deformations are isomorphic, they are obtained from the same sheaf $(\Id_X\times \mathsf{p}_A)^*(\tensor[^{\tau^n}]{\mathscr{F}}{})$ by the same modifications. This defines a morphism $\xi_A^{(n)}:G_1^{(n)}(A)\to G_2^{(n)}(A)$ which is functorial in $A$. Isomorphisms (\ref{isomorphism of fiber at zero}) and (\ref{isomorphism of fiber at pole}) show that $\xi_A^{(n)}=\xi_A^{(m)}$ for all $m\ge n$. Thus we get a morphism of functors $\xi:G_1\to G_2$.

\begin{Proposition}\label{infinitesimal neighborhood of moduli space}
The morphism of functors $\xi:G_1\to G_2$ defined above is an isomorphism.
\end{Proposition}
\begin{proof}
Let $A\in \ArtAlg_E^{(n)}$.

To prove that $\xi_A^{(n)}$ is injective, it suffices to show that in the above construction the modifications for the pair $((\Id_X\times \mathsf{p}_A)^*(\tensor[^{\tau^n}]{\mathscr{F}}{}),\mathscr{F})$ at $\alpha\bcirc\mathsf{p}_A$ and $\beta\bcirc\mathsf{p}_A$ determines the modifications at other places. For the upper modifications, this follows from the fact that the upper modification $\tensor[^{\tau^{i+1}}]{\mathscr{F}}{}\hookrightarrow(\tensor[^{\tau^i}]{\mathscr{F}}{}+\tensor[^{\tau^{i+1}}]{\mathscr{F}}{})$ at $\Gamma_{\Fr^i_X\bcirc \alpha\bcirc\mathsf{p}_A}$ is obtained from the upper modification $\tensor[^\tau]{\mathscr{F}}{}\hookrightarrow(\mathscr{F}+\tensor[^\tau]{\mathscr{F}}{})$ at $\alpha\bcirc\mathsf{p}_A$ by the pullback $(\Phi_A^*)^i$. The proof for the lower modifications is similar.

To prove that $\xi_A^{(n)}$ is surjective, for any morphism $f^\natural:\Spec A\to P^\natural$ such that $f\bcirc e_A=p^\natural$, we construct a deformation of the given shtuka to $\Spec A$ with zero $\alpha\bcirc\mathsf{p}_A$ and pole $\beta\bcirc\mathsf{p}_A$ such that the corresponding morphism is $f^\natural$.

Put $\mathscr{E}_i=(\Id_X\times \mathsf{p}_A)^*(\tensor[^{\tau^i}]{\mathscr{F}}{})$. Write $f^\natural=f^\sharp\times f^\flat:\Spec A\to P^\sharp\times P^\flat$.

Let $\mathscr{E}_i\hookrightarrow\mathscr{E}^\sharp_i$ be the upper modification at $\alpha\bcirc\mathsf{p}_A$ corresponding to $f^\sharp$. Let $\mathscr{E}_i\hookleftarrow\mathscr{E}^{\flat}_i$ be the lower modification at $\beta\bcirc\mathsf{p}_A$ corresponding to $f^\flat$. Then $\mathscr{E}_n\hookrightarrow(\Phi_A^*)^i \mathscr{E}_{n-i}^\sharp$ is an upper modification at $\Fr_X^i\bcirc\alpha\bcirc\mathsf{p}_A$, and $\mathscr{E}_n\hookleftarrow(\Phi_A^*)^i \mathscr{E}_{n-i}^\flat$ is a lower modification at $\Fr_X^i\bcirc\beta\bcirc\mathsf{p}_A$. Let $\mathscr{E}_n^\natural$ be the sheaf obtained from $\mathscr{E}_n$ by upper modifications at $\Fr_X^i\bcirc\alpha\bcirc\mathsf{p}_A$ according to $\mathscr{E}_n\hookrightarrow(\Phi_A^*)^i \mathscr{E}^{\sharp}_{n-i}$ and lower modifications at $\Fr_X^i\bcirc\beta\bcirc\mathsf{p}_A$ according to $\mathscr{E}_n\hookleftarrow(\Phi_A^*)^i \mathscr{E}^{\flat}_{n-i}$ for $i=0,1,\dots,n-1$.

Then $\Phi_A^*\mathscr{E}_n^\natural$ is obtained from $\mathscr{E}_{n+1}$ by upper modifications at $\Fr_X^i\bcirc\alpha\bcirc\mathsf{p}_A$ and lower modifications at $\Fr_X^i\bcirc\beta\bcirc\mathsf{p}_A$ for $i=1,2,\dots,n$. Since $A\in \ArtAlg_E^{(n)}$, the modifications for the pair $(\mathscr{E}_{n+1},\Phi_A^*\mathscr{E}_n^\natural)$ and the pair $(\mathscr{E}_{n+1},\mathscr{E}_n)$ coincide at $\Fr^n_X\bcirc\alpha\bcirc\mathsf{p}_A$ and $\Fr^n_X\bcirc\beta\bcirc\mathsf{p}_A$. From our construction, the modifications for the pair $(\mathscr{E}_{n+1},\Phi_A^*\mathscr{E}_n^\natural)$ and the pair $(\mathscr{E}_{n},\mathscr{E}_n^\natural)$ coincide at $\Fr_X^i\bcirc\alpha\bcirc\mathsf{p}_A$ and $\Fr_X^i\bcirc\beta\bcirc\mathsf{p}_A$ for $i=1,2,\dots,n-1$. Therefore, $\mathscr{E}_n^\natural$ is obtained from $\Phi^*_A\mathscr{E}_n^\natural$ by an upper modification defined by $f^\sharp $ and a lower modification defined by $f^\flat$.
\end{proof}

\section{Multiplicity: Infinitesimal Calculation}
In this section, we use deformation theory to study cohomology of shtukas on infinitesimal neighborhoods of a point on the moduli stack. The goal is to prove \cref{deformation in both directions}.

\subsection{Notation}
Fix a perfect field $E$ over $\mathbb{F}_q$ and fix two morphisms $\alpha,\beta:\Spec E\to X$ satisfying condition (\ref{Frobenius shifts of zero and pole mutually disjoint}). For a scheme $S$ over $\Spec E$, we denote $\mathsf{p}_S:S\to\Spec E$ to be the structure morphism.

Let $\mathscr{B}$ be an invertible sheaf on $X\otimes E$ equipped with an isomorphism $\Phi_E^*\mathscr{B}\xrightarrow{\sim}\mathscr{B}$. We put $h=h^1(X\otimes E,\mathscr{B})$. Let $n$ denote a positive integer greater than or equal to $h$.

Fix a shtuka $\mathscr{F}$ over $\Spec E$ of rank $d$ with zero $\alpha$ and pole $\beta$ such that $\chi(\mathscr{F})=0$. Assume that $\mathscr{F}/\mathscr{F}^{\I}\cong \mathscr{B}$ and $\mathscr{F}^{\II}=0$, where $\mathscr{F}^{\I}$ and $\mathscr{F}^{\II}$ are the subsheaves of $\mathscr{F}$ as in \cref{maximal trivial sub and maximal trivial quotient}.

Recall that the category $\ArtAlg_E^{(n)}$ is defined in \cref{Section: the category of Artinian E-algebras for deformation}. Let $A\in \ArtAlg_E^{(n)}$, and let $\widetilde{\mathscr{F}}$ be a deformation of the given shtuka to $\Spec A$ with zero $\alpha\bcirc\mathsf{p}_A$ and pole $\beta\bcirc\mathsf{p}_A$.

Put $\mathscr{E}_i=(\Id_X\times\mathsf{p}_A)^*(\tensor[^{\tau^i}]{\mathscr{F}}{})$ for $i=0,1,\dots,n$. Then $\widetilde{\mathscr{F}}$ is obtained from $\mathscr{E}_n$ by upper modifications at $\alpha\bcirc\mathsf{p}_A,\Fr_X\bcirc\alpha\bcirc\mathsf{p}_A,\dots,\Fr_X^{n-1}\bcirc\alpha\bcirc\mathsf{p}_A$ and lower modifications at $\beta\bcirc\mathsf{p}_A,\Fr_X\bcirc\beta\bcirc\mathsf{p}_A,\dots,\Fr_X^{n-1}\bcirc\beta\bcirc\mathsf{p}_A$.

Denote $P^\sharp=\mathbb{P}^\vee(\tensor[^\tau]{\mathscr{F}}{}(\alpha)/\tensor[^\tau]{\mathscr{F}}{})$, $P^\flat=\mathbb{P}(\tensor[^\tau]{\mathscr{F}}{}/\tensor[^\tau]{\mathscr{F}}{}(-\beta))$, $P^\natural=P^\sharp\times P^\flat$.

Let $f^\sharp:\Spec A\to P^\sharp$ be the morphism corresponding to the upper modification of $\mathscr{E}_n$ at $\alpha\bcirc\mathsf{p}_A$, and $f^\flat:\Spec A\to P^\flat$ be the morphism corresponding to the lower modification of $\mathscr{E}_n$ at $\beta\bcirc\mathsf{p}_A$. Put $f^\natural=f^\sharp\times f^\flat$.

Let $\mathscr{E}^\sharp_i$ be the upper modification of $\mathscr{E}_i$ at $\alpha\bcirc\mathsf{p}_A$ defined by $f^\sharp$. Let $\mathscr{E}^{\flat}_i$ be the lower modification of $\mathscr{E}_i$ at $\beta\bcirc\mathsf{p}_A$ defined by $f^\flat$.

For $i=0,1,\dots,n-1$, let $\mathscr{E}_n\hookrightarrow\mathscr{E}_n^{\sharp i}$ be the upper modification at $\Fr_X^i\bcirc\alpha\bcirc\mathsf{p}_A$ according to $(\mathscr{E}_n,\widetilde{\mathscr{F}})$, and $\mathscr{E}_n\hookleftarrow\mathscr{E}_n^{\flat i}$ be the lower modification of $\mathscr{E}_n$ at $\Fr_X^i\bcirc\beta\bcirc\mathsf{p}_A$ according to $(\mathscr{E}_n,\widetilde{\mathscr{F}})$.

\cref{infinitesimal neighborhood of moduli space} shows that $\mathscr{E}_n^{\sharp i}\cong (\Phi_A^*)^i\mathscr{E}_{n-i}^\sharp$ and $\mathscr{E}_n^{\flat i}\cong (\Phi_A^*)^i\mathscr{E}_{n-i}^\flat$.

Let $\mathscr{E}_n^+=\sum_{i=0}^{n-1}\mathscr{E}_n^{\sharp i}$ and $\mathscr{E}_n^-=\bigcap_{i=0}^{n-1}\mathscr{E}_n^{\flat i}$. Put $\mathscr{E}_n^\pm=\widetilde{\mathscr{F}}$. So $\mathscr{E}_n^\pm$ is the unique subsheaf of $\mathscr{E}_n^+$ such that $\mathscr{E}_n^\pm + \mathscr{E}_n=\mathscr{E}_n^+$ and $\mathscr{E}_n^\pm\cap\mathscr{E}_n=\mathscr{E}_n^-$.

\subsection{Decrease of $H^1$}
\begin{Lemma}\label{injectivity of the boundary map for generic point of X}
Let $\mathscr{B}_k=\mathscr{B}(\sum_{i=0}^{k-1}\Fr_X^i\bcirc\alpha)$. Then the coboundary map $\delta:H^0(X\otimes E, \mathscr{B}_k/\mathscr{B})\to H^1(X\otimes E,\mathscr{B})$ is injective when $1\le k\le h= h^1(X\otimes E, \mathscr{B})$.
\end{Lemma}
\begin{proof}
Let $m$ be the smallest positive integer such that $h^0(\mathscr{B}_m)>h^0(\mathscr{B})$. Choose $s\in H^0(X\otimes E,\mathscr{B}_m)$ such that $s\notin H^0(X\otimes E,\mathscr{B})$. Then $s$ has nonzero image in $H^0(X\otimes E, \mathscr{B}_m/\mathscr{B}_{m-1})$. The isomorphism $\Phi_E^*\mathscr{B}\xrightarrow{\sim}\mathscr{B}$ induces for each $j\ge 0$  an isomorphism $(\Phi_E^*)^j\mathscr{B}_m\xrightarrow{\sim}\mathscr{B}(\sum_{i=j}^{m+j-1}\Fr_X^i\bcirc\alpha)\subset \mathscr{B}_{m+j}$. Since $\Gamma_{\Fr_X^{i}\bcirc\alpha}$ are mutually disjoint for all $i$ by condition (\ref{Frobenius shifts of zero and pole mutually disjoint}), the element $(\Phi_E^*)^j s\in H^0(X\otimes E,(\Phi_E^*)^j\mathscr{B}_m)\subset H^0(X\otimes E,\mathscr{B}_{m+j})$ has nonzero image in $H^0(X\otimes E, \mathscr{B}_{m+j}/\mathscr{B}_{m+j-1})$. Therefore, for all $j\ge 0$, the elements $s,\Phi_E^* s,\dots,(\Phi_E^*)^j s$  in $H^0(X\otimes E,\mathscr{B}_{m+j})$ are linearly independent, and their linear span has zero intersection with $H^0(X\otimes E,\mathscr{B})$. Thus $h^0(\mathscr{B}_{m+j})= h^0(\mathscr{B})+j+1$. When $j$ is sufficiently large, we have
\[\chi(\mathscr{B})+m+j=\chi(\mathscr{B}_{m+j})=h^0(\mathscr{B}_{m+j})=h^0(\mathscr{B})+j+1=\chi(\mathscr{B})+h+j+1\]
Thus $m=h+1$. This means that $h^0(\mathscr{B})=h^0(\mathscr{B}_k)$ when $1\le k\le h$. The statement follows.
\end{proof}

\subsection{Identification of the boundary maps}
By isomorphism (\ref{isomorphism of fiber at zero}) we can identify ${\pi_E}_*(\tensor[^{\tau^i}]{\mathscr{F}}{}(\alpha)/\tensor[^{\tau^i}]{\mathscr{F}}{})$ for all $i>0$. Since $E$ is perfect, $\alpha$ and $\beta$ satisfy condition (\ref{Frobenius shifts of zero and pole mutually disjoint}), $\chi(\mathscr{F})=0$, $\mathscr{F}/\mathscr{F}^{\I}\cong\mathscr{B}$, and $\mathscr{F}^{\II}=0$, we have an isomorphism $R^1{\pi_E}_*\mathscr{F}\xrightarrow{\sim}R^1{\pi_E}_*\mathscr{B}$ from \cref{cohomology of shtukas over a field}. From the isomorphism $\Phi^*_E\mathscr{B}\xrightarrow{\sim}\mathscr{B}$, we get isomorphisms $R^1{\pi_E}_*\tensor[^{\tau^i}]{\mathscr{F}}{}\xrightarrow{\sim} R^1{\pi_E}_*\mathscr{B}$ for all $i>0$.
\begin{Lemma}\label{identification of coboundary maps}
The morphisms ${\pi_E}_*(\tensor[^{\tau^i}]{\mathscr{F}}{}(\alpha)/\tensor[^{\tau^i}]{\mathscr{F}}{})\to R^1{\pi_E}_*(\mathscr{B})$ don't depend on $i$ for $i>0$ and they have rank 1. Their kernels are the same as the kernel of the morphism ${\pi_E}_*(\tensor[^{\tau}]{\mathscr{F}}{}(\alpha)/\tensor[^{\tau}]{\mathscr{F}}{})\to{\pi_E}_*(\mathscr{B}(\alpha)/\mathscr{B})$.
\end{Lemma}
\begin{proof}
The morphism ${\pi_E}_*(\tensor[^{\tau^i}]{\mathscr{F}}{}(\alpha)/\tensor[^{\tau^i}]{\mathscr{F}}{})\to R^1{\pi_E}_*(\mathscr{B})$ factors as ${\pi_E}_*(\tensor[^{\tau^i}]{\mathscr{F}}{}(\alpha)/\tensor[^{\tau^i}]{\mathscr{F}}{})\to {\pi_E}_*(\mathscr{B}(\alpha)/\mathscr{B})\to
R^1{\pi_E}_*\mathscr{B}$.
The morphism ${\pi_E}_*(\tensor[^{\tau^i}]{\mathscr{F}}{}(\alpha)/\tensor[^{\tau^i}]{\mathscr{F}}{})\to {\pi_E}_*(\mathscr{B}(\alpha)/\mathscr{B})$ is surjective since $\tensor[^{\tau^i}]{\mathscr{F}}{}\to \mathscr{B}$ is surjective. The $E$-vector space ${\pi_E}_*(\mathscr{B}(\alpha)/\mathscr{B})$ has dimension 1 since $\rank\mathscr{B}=1$. The morphism ${\pi_E}_*(\mathscr{B}(\alpha)/\mathscr{B})\to
R^1{\pi_E}_*\mathscr{B}$ is injective by \cref{injectivity of the boundary map for generic point of X}. Therefore the morphisms in question have rank 1.

Let $\mathscr{G}=\mathscr{F}+\Phi_E^*\mathscr{F}$. Since $\mathscr{G}$ is obtained from $\mathscr{F}$ by partial Frobenius $F_2$, \cref{cohomology of shtukas over a field} gives an isomorphism $R^1{\pi_E}_*(\mathscr{G})\xrightarrow{\sim} R^1{\pi_E}_*(\mathscr{B})$. Hence we get isomorphisms $R^1{\pi_E}_*(\tensor[^{\tau^i}]{\mathscr{G}}{})\xrightarrow{\sim} R^1{\pi_E}_*(\mathscr{B})$ for all $i\ge 0$.

Let $i\ge 1$. The diagram $ \tensor[^{\tau^i}]{\mathscr{F}}{}\hookrightarrow\tensor[^{\tau^i}]{\mathscr{G}}{}\hookleftarrow
\tensor[^{\tau^{i+1}}]{\mathscr{F}}{}$ induces the following commutative diagram
\[\begin{tikzcd}
  {\pi_E}_*(\tensor[^{\tau^i}]{\mathscr{F}}{}(\alpha)/\tensor[^{\tau^i}]{\mathscr{F}}{})\arrow[r]\arrow[d,"\cong"]
  & R^1{\pi_E}_*(\tensor[^{\tau^i}]{\mathscr{F}}{}) \arrow[r,"\cong"]\arrow[d] & R^1{\pi_E}_*(\mathscr{B})\arrow[d,equal] \\
  {\pi_E}_*(\tensor[^{\tau^i}]{\mathscr{G}}{}(\alpha)/\tensor[^{\tau^i}]{\mathscr{G}}{})\arrow[r]
  & R^1{\pi_E}_*(\tensor[^{\tau^i}]{\mathscr{G}}{})\arrow[r,"\cong"] & R^1{\pi_E}_*(\mathscr{B})\\
  {\pi_E}_*(\tensor[^{\tau^{i+1}}]{\mathscr{F}}{}(\alpha)/\tensor[^{\tau^{i+1}}]{\mathscr{F}}{})\arrow[r]\arrow[u,"\cong"']
  & R^1{\pi_E}_*(\tensor[^{\tau^{i+1}}]{\mathscr{F}}{}) \arrow[r,"\cong"]\arrow[u] & R^1{\pi_E}_*(\mathscr{B})\arrow[u,equal]
\end{tikzcd}\]
The isomorphisms in the left column are from the isomorphisms
\[\tensor[^{\tau^i}]{\mathscr{F}}{}(\alpha)/\tensor[^{\tau^i}]{\mathscr{F}}{}\cong
\tensor[^{\tau^i}]{\mathscr{G}}{}(\alpha)/\tensor[^{\tau^i}]{\mathscr{G}}{}
\cong\tensor[^{\tau^{i+1}}]{\mathscr{F}}{}(\alpha)/\tensor[^{\tau^{i+1}}]{\mathscr{F}}{}\]
Therefore, the first row and the last row have the same kernel and image. This proves the first statement.

The second statment follows from injectivity of the morphism ${\pi_E}_*(\mathscr{B}(\alpha)/\mathscr{B})\to
R^1{\pi_E}_*\mathscr{B}$ and the identifications of the two morphisms ${\pi_E}_*(\tensor[^{\tau^i}]{\mathscr{F}}{}(\alpha)/\tensor[^{\tau^i}]{\mathscr{F}}{})\to R^1{\pi_E}_*(\mathscr{B})$ and ${\pi_E}_*(\tensor[^{\tau}]{\mathscr{F}}{}(\alpha)/\tensor[^{\tau}]{\mathscr{F}}{})\to R^1{\pi_E}_*(\mathscr{B})$
\end{proof}

\subsection{Deformation in the direction of major importance}\label{(Section)deformation in the direction of major importance}
\begin{Lemma}\label{cokernel of varying lines}
Let $V,W_0$ be two vector spaces over $E$ and assume $\dim_E W_0=1$. Let $\delta:V\to W_0$ be a nonzero linear map. Put $P=\ProjSym V^*$. Let $S$ be a scheme over $E$ with a morphism $f:S\to P$, and $\mathscr{L}=f^*\mathscr{O}_P(-1)$. Put $\mathscr{V}=\mathsf{p}_S^*V,\mathscr{W}_0=\mathsf{p}_S^*W_0$, Let $\Delta:\mathscr{L}\to \mathscr{V}\xrightarrow{\mathsf{p}_S^*\delta}\mathscr{W}_0$ be the composition induced by the canonical morphism and $\delta$. Then $\im\Delta=\mathscr{W}_0\otimes \mathscr{I}_{f^{-1}(H_\delta)}$, where $H_\delta$ is the hyperplane of $P$ corresponding to $\ker \delta$ and $\mathscr{I}_{f^{-1}(H_\delta)}\subset\mathscr{O}_S$ is the ideal sheaf of $f^{-1}(H_\delta)$.
\end{Lemma}
\begin{proof}
It suffices to prove the lemma for $S=P, f=\Id$. In this case, since $\mathscr{L}\to\mathscr{V}$ is the canonical morphism, $\Delta$ identifies $\mathscr{L}$ as the subsheaf $\mathscr{W}_0\otimes \mathscr{I}_{H_\delta}$ of $\mathscr{W}_0$.
\end{proof}

Denote $H^\sharp$ to be the hyperplane of $P^\sharp$ corresponding to the kernel of the morphism ${\pi_E}_*(\tensor[^{\tau}]{\mathscr{F}}{}(\alpha)/\tensor[^{\tau}]{\mathscr{F}}{})\to{\pi_E}_*(\mathscr{B}(\alpha)/\mathscr{B})$. Recall that $f^\sharp:\Spec A\to P^\sharp$ is the morphism corresponding to the upper modification of the pair $(\mathscr{E}_n,\widetilde{\mathscr{F}})$ at $\alpha\bcirc\mathsf{p}_A$.

The following proposition deals with deformation in the direction of major importance. Recall that we assume $n\ge h$.

\begin{Proposition}\label{deformation in the major direction}
We have an isomorphism $R^1{\pi_A}_*\mathscr{E}_n^+\cong\bigoplus_{i=0}^{h-1}\mathscr{O}_A/\mathscr{J}^{q^i}$, where $\mathscr{J}$ is the ideal sheaf of ${f^\sharp}^{-1}(H^\sharp)$.
\end{Proposition}
\begin{proof}
Let $W=R^1{\pi_E}_*\mathscr{B}$, $W_0=\im( {\pi_E}_*(\mathscr{B}(\alpha)/\mathscr{B})\to R^1{\pi_E}_*\mathscr{B})$. We have an isomorphism $\Fr_E^*W\cong W$ induced by $\Phi_E^*\mathscr{B}\cong \mathscr{B}$. For $i=1,2,\dots,n-1$, put $W_i=(\Fr_E^*)^i W_0\subset W$. Let $\mathscr{W}=\mathsf{p}_A^*W, \mathscr{W}_i=\mathsf{p}_A^*W_i$.

For $j=1,2,\dots,n$, let $V^j={\pi_E}_*(\tensor[^{\tau^j}]{\mathscr{F}}{}(\alpha)/\tensor[^{\tau^j}]{\mathscr{F}}{})$. Let $\delta^j:V^j\to W$ be the morphism ${\pi_E}_*(\tensor[^{\tau^j}]{\mathscr{F}}{}(\alpha)/\tensor[^{\tau^j}]{\mathscr{F}}{})\to R^1{\pi_E}_*\mathscr{B}$. By \cref{identification of coboundary maps}, we have $\rank\delta^j=1$ and  $\im\delta^j=W_0$ for $j=1,2,\dots,n$. By flat base change $\mathscr{W}=R^1{\pi_A}_*\mathscr{E}_j$. Since $\mathscr{E}_j^\sharp$ is the upper modification of $\mathscr{E}_j$ at $\alpha\bcirc\mathsf{p}_A$ corresponding to $f^\sharp$, we have ${\pi_A}_*(\mathscr{E}_j^\sharp/\mathscr{E}_j)={f^\sharp}^*\mathscr{O}_{P^\sharp}(-1)$, ${\pi_A}_*(\mathscr{E}_j(\alpha\bcirc\mathsf{p}_A)/\mathscr{E}_j)=\mathsf{p}_A^*V^j$ and the morphism ${\pi_A}_*(\mathscr{E}_j^\sharp/\mathscr{E}_j)\to {\pi_A}_*(\mathscr{E}_j(\alpha\bcirc\mathsf{p}_A)/\mathscr{E}_j)$ is the canonical morphism. Thus we can apply \cref{cokernel of varying lines} to deduce that the image of the coboundary map ${\pi_A}_*(\mathscr{E}_j^\sharp/\mathscr{E}_j)\to R^1{\pi_A}_*\mathscr{E}_j$ is $\mathscr{W}_0\otimes \mathscr{I}_{{f^\sharp}^{-1}(H_{\delta^j})}$, where $H_{\delta^j}$ is the hyperplane of $P^\sharp$ corresponding to the $\ker\delta^j$, and $\mathscr{I}_{{f^\sharp}^{-1}(H_{\delta^j})}\subset\mathscr{O}_A$ is the ideal sheaf of ${f^\sharp}^{-1}(H_{\delta^j})$. By \cref{identification of coboundary maps}, we see that $H_{\delta^j}=H^\sharp$ for $j=1,2,\dots,n$. So $R^1{\pi_A}_*\mathscr{E}_j^\sharp\cong \mathscr{W}/\mathscr{W}_0\otimes \mathscr{J}$.

Recall that $\mathscr{E}_n^{\sharp i}\cong (\Phi_A^*)^i\mathscr{E}_{n-i}^\sharp$ by \cref{infinitesimal neighborhood of moduli space}. Since $\dim X=1$, by base change for cohomology, we have
\begin{align*}
&R^1{\pi_A}_*\mathscr{E}_n^{\sharp i}= R^1{\pi_A}_*(\Phi_A^*)^i\mathscr{E}_{n-i}^\sharp=(\Fr_A^*)^i R^1{\pi_A}_*\mathscr{E}_{n-i}^\sharp \\=&(\Fr_A^*)^i(\mathscr{W}/\mathscr{W}_0\otimes \mathscr{J})=\mathscr{W}/\mathscr{W}_i\otimes \mathscr{J}^{q^i}
\end{align*}
Since $R^1{\pi_A}_*$ is right exact, it commutes with push out. So $\mathscr{E}_n^+=\sum_{i=1}^{n-1}\mathscr{E}_n^{\sharp i}$ implies that
\[R^1{\pi_A}_*\mathscr{E}_n^+ =\mathscr{W}/\sum_{i=0}^{n-1}\mathscr{W}_i\otimes \mathscr{J}^{q^i}\]
Note that $W_i=(\Fr_E^*)^i W_0=\im({\pi_E}_*(\mathscr{B}(\Fr_X^i\bcirc\alpha)/\mathscr{B})\to R^1{\pi_E}_*\mathscr{B})$. \cref{injectivity of the boundary map for generic point of X} implies that $W=\bigoplus_{i=0}^{h-1}W_i$. So $\mathscr{W}=\bigoplus_{i=0}^{h-1}\mathscr{W}_i$. For $j>h$, we have
\[\mathscr{W}_j\otimes \mathscr{J}^{q^j}\subset \mathscr{W}\otimes\mathscr{J}^{q^j}\subset\sum_{i=0}^{h-1}\mathscr{W}_i\otimes \mathscr{J}^{q^i}\]
Hence the assumption $n\ge h$ implies that
\[\sum_{i=0}^{n-1}\mathscr{W}_i\otimes \mathscr{J}^{q^i}=\sum_{i=0}^{h-1}\mathscr{W}_i\otimes \mathscr{J}^{q^i}\]
Therefore,
\[R^1{\pi_A}_*\mathscr{E}_n^+=\bigoplus_{i=0}^{h-1}\mathscr{W}_i/\mathscr{W}_i\otimes \mathscr{J}^{q^i}
\cong \bigoplus_{i=0}^{h-1}\mathscr{O}_A/\mathscr{J}^{q^i}\]
\end{proof}

\subsection{Deformation in the direction of minor importance}\label{(Section)deformation in the direction of minor importance}
The following lemma deals with deformation in the direction of minor importance. Recall that we assume $n\ge h$.
\begin{Lemma}\label{deformation in the minor direction}
The inclusion $\mathscr{E}_n^\pm\to\mathscr{E}_n^+$ induces an isomorphism $R^1{\pi_A}_*\mathscr{E}_n^\pm\to R^1{\pi_A}_*\mathscr{E}_n^+$.
\end{Lemma}
\begin{proof}
Consider the following commutative diagram
\[\begin{tikzcd}
{\pi_A}_*(\mathscr{E}_n^+/\mathscr{E}_n^\pm) \arrow[r,"d_n^\pm"] & R^1{\pi_A}_*\mathscr{E}^\pm_n \arrow[r,"j_n^\pm"] & R^1{\pi_A}_*\mathscr{E}^+_n \arrow[r] & 0 \\
{\pi_A}_*(\mathscr{E}_n/\mathscr{E}_n^-) \arrow[r,"d_n^-"]\arrow[u,"i_n"] & R^1{\pi_A}_*\mathscr{E}^-_n \arrow[r, "j_n^-"]\arrow[u] & R^1{\pi_A}_*\mathscr{E}_n \arrow[r]\arrow[u] & 0
\end{tikzcd}\]
where the horizontal sequences are exact. Since $\alpha\bcirc\mathsf{p}_A,\Fr_X\bcirc\alpha\bcirc\mathsf{p}_A,\dots,\Fr_X^{n-1}\bcirc\alpha\bcirc\mathsf{p}_A$ and $\beta\bcirc\mathsf{p}_A,\Fr_X\bcirc\beta\bcirc\mathsf{p}_A,\dots,\Fr_X^{n-1}\bcirc\beta\bcirc\mathsf{p}_A$ have disjoint graphs, the map $i_n$ is an isomorphism. By commutativity of the left square, to prove that $d_n^\pm=0$, it suffices to prove that $d_n^-=0$.

We prove that it is indeed the case if $n=h$.

It suffices to prove that $j_h^-$ is an isomorphism. Since $\mathscr{E}_h/\mathscr{E}_h^-$ is torsion, $j_h^-$ is surjective. Recall that $\mathscr{E}_h=(\Id_X\times\mathsf{p}_A)^*(\tensor[^{\tau^h}]{\mathscr{F}}{})$. By flat base change, $R^1{\pi_A}_*\mathscr{E}_h\cong \mathsf{p}_A^*R^1{\pi_E}_*(\tensor[^{\tau^h}]{\mathscr{F}}{})$ is free. Hence it remains to show that $e_A^*R^1{\pi_A}_*\mathscr{E}_h^-$ and $e_A^*R^1{\pi_A}_*\mathscr{E}_h$ have the same dimension over $E$, where $e_A:\Spec E\to\Spec A$ is the closed point. By construction, $e_A^*\mathscr{E}_h=\tensor[^{\tau^h}]{\mathscr{F}}{}, e_A^*\mathscr{E}_h^-=\mathscr{F}\cap\tensor[^{\tau^h}]{\mathscr{F}}{}$. Since $\dim X=1$, by base change of cohomology, $\dim_E e_A^*R^1{\pi_A}_*\mathscr{E}_h=h^1(\tensor[^{\tau^h}]{\mathscr{F}}{}), \dim_E e_A^*R^1{\pi_A}_*\mathscr{E}_h^-=h^1(\mathscr{F}\cap\tensor[^{\tau^h}]{\mathscr{F}}{})$. Since $\chi(\tensor[^{\tau^h}]{\mathscr{F}}{})=0, \chi(\mathscr{F}\cap\tensor[^{\tau^h}]{\mathscr{F}}{})=-h,\mathscr{F}^{\II}=0,\mathscr{F}/\mathscr{F}^{\I}\cong\mathscr{B}$, we have $h^1(\tensor[^{\tau^h}]{\mathscr{F}}{})=h^1(\mathscr{F}\cap\tensor[^{\tau^h}]{\mathscr{F}}{})=h$ by \cref{h^0 and h^1 of a shtuka} and \cref{Remark: exact sequence for reducible shtuka}.

Now we prove that $j_n^\pm$ is an isomorphism for all $n\ge h$.

Recall that $P^\natural=P^\sharp\times P^\flat$. As in \cref{Section: the category of Artinian E-algebras for deformation}, let $p^\natural$ be the $E$-point of $P^\natural$ corresponding to the upper and lower modifications in the shtuka $\mathscr{F}$. For $j\ge 0$, let $\Spec A_j$ be the $(q^j-1)$-th infinitesimal neighborhood of $p^\natural$ in $P^\natural$. By base change of cohomology, the statement follows from the special case when $A=A_n$ and $\widetilde{\mathscr{F}}$ is the deformation of $\mathscr{F}$ corresponding to the natural closed immersion $f_n:\Spec A_n\to P^\natural$.

Recall that $H^\sharp$ denotes the hyperplane of $P^\sharp$ corresponding to the kernel of the morphism ${\pi_E}_*(\tensor[^{\tau}]{\mathscr{F}}{}(\alpha)/\tensor[^{\tau}]{\mathscr{F}}{})\to{\pi_E}_*(\mathscr{B}(\alpha)/\mathscr{B})$. Let $x\in A_n$ be a generator of the ideal of $A_n$ corresponding to $f_n^{-1}(H^\sharp\times P^\flat)$. Then \cref{deformation in the major direction} implies that
\[R^1{\pi_{A_n}}_*\mathscr{E}_n^+\cong \bigoplus_{i=0}^{h-1}A_n/x^{q^i}\]
By our construction, $\mathscr{E}_n^\pm$ and $\mathscr{E}_h^\pm$ are canonically isomorphic on $X\times\Spec A_h$. Thus the case $n=h$ above implies that $j_n^\pm$ induces an isomorphism on $\Spec A_h$. The $A_n$-module $R^1{\pi_{A_n}}_*\mathscr{E}_n^\pm$ is finitely generated since $\mathscr{E}_n^\pm$ is coherent. Applying \cref{isomorphism on smaller neighborhood implies isomorphism}, we deduce that $j_n^\pm$ is an isomorphism.
\end{proof}

\begin{Lemma}\label{isomorphism on smaller neighborhood implies isomorphism}
Let $R=E[[x_0,x_1,\dots,x_k]]$ and $\mathfrak{m}$ be its maximal ideal. Let $R_l=R/\mathfrak{m}^l$ for each $l\ge 0$. Suppose there is a surjective morphism of finitely generated $R_b$-modules $\xi:U\to \bigoplus_{i=1}^m R_b/x_0^{d_i}$ which induces an isomorphism on $\Spec R_a$ for some $a\le b$, and $0<d_i<a$ for all $i=1,2,\dots,m$. Then $\xi$ is an isomorphism.
\end{Lemma}
\begin{proof}
For $i=1,2,\dots, m$, let $\bar{e}_i$ be the element of $\bigoplus_{i=1}^m R_b/x_0^{d_i}$ which is 1 on the $i$-th component and 0 elsewhere. Pick $u_i\in U$ such that $\xi(u_i)=\bar{e}_i$. Since $U$ is finitely generated, by Nakayama's lemma, $u_1,u_2,\dots,u_m$ generate $U$. Consider morphisms of $R_b$-modules
\[\begin{tikzcd}
  R_b^{\oplus m}\arrow[r,"\zeta"]& U\arrow[r,"\xi"]&\bigoplus_{i=1}^m R_b/x_0^{d_i}
\end{tikzcd}\]
where $\zeta$ is given by $\zeta(e_i)=u_i$ on a basis $\{e_1,\dots,e_m\}$ of $R_b^{\oplus m}$. Since $\xi$ induces an isomorphism on $\Spec R_a$, we have $x_0^{d_i}e_i\in \ker\zeta+\mathfrak{m}^a R_b^{\oplus m}$ for each $i$. Hence for each $i$ there exists $z_i=\sum_{j=1}^m z_{ij} e_j\in \ker\zeta$ such that $z_i-x_0^{d_i}e_i\in \mathfrak{m}^a R_b^{\oplus m}$. Since $\ker\zeta\subset\ker\xi\bcirc\zeta=\bigoplus_{j=1}^m R_b\cdot x_0^{d_j}e_j$, we can write $z_{ij}=r_{ij}x_0^{d_j}$ for some $r_{ij}\in R_b$. So $z_i-x_0^{d_i}e_i=(r_{ii}-1)x_0^{d_i}e_i+\sum_{j\ne i}r_{ij}x_0^{d_j}e_j$. Since $a>d_j$ for each $j$, we see that $r_{ii}-1\in \mathfrak{m}$ for all $i$ and $r_{ij}\in\mathfrak{m}$ for all $j\ne i$. Thus $z_i-x_0^{d_i}e_i\in \mathfrak{m}\cdot\ker\xi\bcirc\zeta$ for all $i$. Therefore, we have $\ker\xi\bcirc\zeta=\ker\zeta+\mathfrak{m}\cdot\ker\xi\bcirc\zeta$. By Nakayama's lemma, we deduce that $\ker \zeta=\ker \xi\bcirc\zeta$. Note that $\zeta$ is surjective since $u_1,u_2,\dots,u_m$ generate $U$. The statement follows.
\end{proof}

\subsection{Description of $R^1{\pi_A}_*\widetilde{\mathscr{F}}$}
Combining the results in \cref{(Section)deformation in the direction of major importance,(Section)deformation in the direction of minor importance}, we get the result for deformation in both directions.

Recall that $P^\sharp=\mathbb{P}^\vee(\tensor[^\tau]{\mathscr{F}}{}(\alpha)/\tensor[^\tau]{\mathscr{F}}{})$, $P^\flat=\mathbb{P}(\tensor[^\tau]{\mathscr{F}}{}/\tensor[^\tau]{\mathscr{F}}{}(-\beta))$, $P^\natural=P^\sharp\times P^\flat$. We denoted $H^\sharp$ to be the hyperplane of $P^\sharp$ corresponding to the kernel of the morphism ${\pi_E}_*(\tensor[^{\tau}]{\mathscr{F}}{}(\alpha)/\tensor[^{\tau}]{\mathscr{F}}{})\to{\pi_E}_*(\mathscr{B}(\alpha)/\mathscr{B})$. Recall that $\mathscr{E}_n^\pm=\widetilde{\mathscr{F}}$, and $f^\natural:\Spec A\to P^\natural=P^\sharp\times P^\natural$ is the morphism corresponding to the upper modification of the pair $(\mathscr{E}_n,\widetilde{\mathscr{F}})$ at $\alpha\bcirc\mathsf{p}_A$ and the lower modification of the pair $(\mathscr{E}_n,\widetilde{\mathscr{F}})$ at $\beta\bcirc\mathsf{p}_A$.

\begin{Proposition}\label{deformation in both directions}
  We have an isomorphism $R^1{\pi_A}_*\widetilde{\mathscr{F}}\cong\bigoplus_{i=0}^{h-1}\mathscr{O}_A/\mathscr{K}^{q^i}$, where $\mathscr{K}$ is the ideal sheaf of ${f^\natural}^{-1}(H^\sharp\times P^\flat)$.
\end{Proposition}

\begin{proof}
  The statement follows from \cref{deformation in the major direction} and \cref{deformation in the minor direction}.
\end{proof}

\section{Multiplicity: Result}
In this section, we use results in the previous section to obtain the multiplicities of cohomology of shtukas at horospherical divisors. The goal is to prove \cref{divisor of determinant of cohomology of shtukas}, which describes the determinant of cohomology of the universal shtuka.
\subsection{Definition and Basic Properties for Horocycles}
Fix an integer $d\ge 2$.

Let  $E$ be an algebraically closed field over $\mathbb{F}_q$. Fix two morphisms $\alpha,\beta:\Spec E\to X$ over $\mathbb{F}_q$ satisfying condition (\ref{Frobenius shifts of zero and pole mutually disjoint}).

For a scheme $S$ over $\Spec E$, let $\mathsf{p}_S:S\to\Spec E$ denote the structure morphism.

For all $r\ge 1$, let $\operatorname{Sht}_E^r$ be the moduli stack of shtukas of rank $r$ over $\Spec E$ with zero $\alpha$ and pole $\beta$.

Let $\mathcal{M}$ be the moduli stack of shtukas $\mathscr{F}$ of rank $d$ over $\Spec E$ with zero $\alpha$ and pole $\beta$ satisfying $\chi(\mathscr{F})=0$.

Let $\mathfrak{Vect}_X^r$ denote the set of isomorphism classes of locally free sheaves of rank $r$ on $X$. In particular, we have $\mathfrak{Vect}_X^1=\Pic(X)$.

\begin{Definition}
A \emph{trivial shtuka} of rank $r$ over an $\mathbb{F}_q$-scheme $S$ is a locally free sheaf $\mathscr{E}$ on $X\times S$ of rank $r$ together with an isomorphism $\Phi_S^*\mathscr{E}\xrightarrow{\sim}\mathscr{E}$.
\end{Definition}

Let $\mathcal{T}^r$ denote the stack classifying trivial shtukas of rank $r$ over $\Spec E$. Theorem 2 of Section 3 of Chapter 1 of~\cite{L.Lafforgue97} says that
\[\mathcal{T}^r=\coprod_{\mathscr{E}\in\mathfrak{Vect}_X^r}\mathcal{T}_\mathscr{E}\]
where $\mathcal{T}_{\mathscr{E}}$ is the quotient stack $[\Spec E/\Aut\mathscr{E}]$.

\begin{Definition}
Given $1\le r\le d-1$ and $\mathscr{E}\in\mathfrak{Vect}_X^r$, we define $\mathcal{R}_{\mathscr{E}}^{\I}$ to be the stack which to every scheme $S$ over $\Spec E$ associates the groupoid of exact sequences
\[\begin{tikzcd}
0 \arrow[r]& \mathscr{A} \arrow[r]& \mathscr{F} \arrow[r]& \mathscr{B} \arrow[r]& 0
\end{tikzcd}\]
where

(i) $\mathscr{A}\in\Sht_E^{d-r}(S), \mathscr{F}\in\mathcal{M}(S),\mathscr{B}\in\mathcal{T}^r(S)$;

(ii) the morphisms $\mathscr{A}\to\mathscr{F}$ and $\mathscr{F}\to\mathscr{B}$ are morphisms of shtukas.
\end{Definition}

\begin{Definition}
Given $1\le r\le d-1$ and $\mathscr{E}\in\mathfrak{Vect}_X^r$, we define $\mathcal{R}_{\mathscr{E}}^{\II}$ to be the stack which to every scheme $S$ over $\Spec E$ associates the groupoid of exact sequences
\[\begin{tikzcd}
0 \arrow[r]& \mathscr{A} \arrow[r]& \mathscr{F} \arrow[r]& \mathscr{B} \arrow[r]& 0
\end{tikzcd}\]
where

(i) $\mathscr{A}\in\mathcal{T}^r(S), \mathscr{F}\in\mathcal{M}(S),\mathscr{B}\in\Sht_E^r(S)$;

(ii) the morphisms $\mathscr{A}\to\mathscr{F}$ and $\mathscr{F}\to\mathscr{B}$ are morphisms of shtukas.
\end{Definition}

The following proposition is an immediate consequence of Theorem 5 of Section 1 of Chapter II of~\cite{L.Lafforgue97}.
\begin{Proposition}
The morphism $\mathcal{R}^{\I}_\mathscr{B}\to \mathcal{M}$ (resp. $\mathcal{R}^{\II}_\mathscr{A}\to\mathcal{M}$) is representable, quasi-finite, non-ramified and separated for any locally free sheaf $\mathscr{B}$ (resp. $\mathscr{A}$) on $X$  with $1\le\rank\mathscr{B}\le d-1$ (resp.  $1\le\rank\mathscr{A}\le d-1$).
\end{Proposition}

For a locally free sheaf $\mathscr{B}$ (resp. $\mathscr{A}$)  on $X$ with $1\le \rank\mathscr{B}\le d-1$ (resp. $1\le\rank\mathscr{A}\le d-1$.), we define the horocycle $\mathcal{Z}^{\I}_\mathscr{B}$ (resp. $\mathcal{Z}^{\II}_\mathscr{A}$) to be the closure of the image of the natural morphism $\mathcal{R}^{\I}_\mathscr{B}\to \mathcal{M}$ (resp $\mathcal{R}^{\II}_\mathscr{A}\to\mathcal{M}$).

 The following proposition is an immediate consequence of Theorem 11 of Section 1 of Chapter II of~\cite{L.Lafforgue97}.

\begin{Proposition}
The morphism
\[\mathcal{R}^{\I}_\mathscr{B}\to \Sht^{d-\rank \mathscr{B}}_E\times \Spec E/\Aut \mathscr{B}\]
\[\text(resp. )\quad \mathcal{R}^{\II}_\mathscr{A}\to \operatorname{Sht}^{d-\rank \mathscr{A}}_E\times \Spec E/\Aut \mathscr{A}\]
is of finite type and smooth of pure relative dimension $\rank \mathscr{B}$ (resp. $\rank\mathscr{A}$).
\end{Proposition}

It follows from the above proposition that $\mathcal{R}^{\I}_\mathscr{B}$ and $\mathcal{R}^{\II}_\mathscr{A}$ are reduced. Hence $\mathcal{Z}^{\I}_\mathscr{B}$  and $\mathcal{Z}^{\II}_\mathscr{A}$ are reduced substacks of $\mathcal{M}$.

Since $\operatorname{Sht}^{r}_E$ has pure dimension $2r-2$ over $\Spec E$ and the morphism $\mathcal{R}^{\I}_\mathscr{B}\to\mathcal{M}$ (resp. $\mathcal{R}^{\II}_\mathscr{A}\to\mathcal{M}$) is quasi-finite for each $\mathscr{B}$ (resp. $\mathscr{A}$), we get the following corollary.

\begin{Corollary}\label{codimension of horocycles of type I or type II}
$\mathcal{Z}^{\I}_\mathscr{B}$ (resp. $\mathcal{Z}^{\II}_\mathscr{A}$) has pure codimension $\rank\mathscr{B}$ (resp. $\rank\mathscr{A}$) in $\mathcal{M}$. In particular, $\mathcal{Z}^{\I}_\mathscr{B}$ (resp. $\mathcal{Z}^{\II}_\mathscr{A}$) is a horospherical divisor when $\mathscr{B}$ (resp. $\mathscr{A}$) has rank one.
\end{Corollary}

\begin{Definition}
Given $d\ge 2, i,j\ge 1, i+j\le d$ and $\mathscr{A}\in\mathfrak{Vect}_X^i,\mathscr{B}\in\mathfrak{Vect}_X^j$, we define $\mathcal{R}^{\I\wedge\II}_{\mathscr{A},\mathscr{B}}$ be the stack which to every scheme $S$ over $\Spec E$ associates the groupoid of the following data:

(i) an exact sequence of shtukas $0\to\mathscr{N}\to\mathscr{F}\to\mathscr{Q}\to 0$, where $\mathscr{N}\in\Sht_E^{d-j}(S), \mathscr{F}\in\mathcal{M}(S), \mathscr{Q}\in\mathcal{T}_{\mathscr{B}}(S)$.

(ii) an exact sequence of shtukas $0\to\mathscr{S}\to\mathscr{N}\to\mathscr{M}\to 0$, where $\mathscr{S}\in\mathcal{T}_{\mathscr{A}}(S), \mathscr{M}\in\Sht_E^{d-i-j}(S)$ and $\mathscr{N}$ is as in (i).

The above data (i), (ii) are equivalent to the following data:

(i') an exact sequence of shtukas $0\to\mathscr{S}\to\mathscr{F}\to\mathscr{L}\to 0$, where $\mathscr{S}\in\mathcal{T}_\mathscr{A}(S),\mathscr{F}\in\mathcal{M}(S),\mathscr{L}\in\Sht_E^{d-i}(S)$;

(ii') an exact sequence of shtukas $0\to\mathscr{M}\to\mathscr{L}\to\mathscr{Q}\to 0$, where $\mathscr{M}\in\Sht_E^{d-i-j}(S),\mathscr{Q}\in\mathcal{T}_\mathscr{B}(S)$ and $\mathscr{L}$ is as in (i').
\end{Definition}

Theorem 11 of Section 1 of Chapter II of~\cite{L.Lafforgue97} implies the following proposition. See also the remark on top of page 103 of~\cite{L.Lafforgue97}.

\begin{Proposition}
The canonical morphism $\mathcal{R}^{\I\wedge\II}_{\mathscr{A},\mathscr{B}}\to\Sht_E^{d-\rank\mathscr{A}-\rank\mathscr{B}}$ is of finite type and smooth of relative dimension $\rank\mathscr{A}+\rank\mathscr{B}$.
\end{Proposition}

\begin{Corollary}\label{dimension of horocycles of type I+II}
$\mathcal{R}^{\I\wedge\II}_{\mathscr{A},\mathscr{B}}$ has dimension $(2d-2)-(\rank\mathscr{A}+\rank\mathscr{B})$ over $\Spec E$.
\end{Corollary}

\begin{Lemma}\label{horospherical divisor of type I not contained in other horocycles of type I}
Let $\mathscr{B}\in\Pic(X)$. Then for any $1\le i\le d-1$, $\mathscr{B}'\in\mathfrak{Vect}_X^i$ such that $\mathscr{B}'\not\cong\mathscr{B}$, the horocycle $\mathcal{Z}^{\I}_{\mathscr{B}'}$ does not contain any irreducible component of the horocycle $\mathcal{Z}^{\I}_\mathscr{B}$.
\end{Lemma}
\begin{proof}
Let $\mathcal{Y}$ be an irreducible component of $\mathcal{Z}^{\I}_\mathscr{B}$. Suppose $\mathcal{Y}$ is contained in $\mathcal{Z}^{\I}_{\mathscr{B}'}$ for some $\mathscr{B}'\in\mathfrak{Vect}_X^i$ such that $\mathscr{B}'\not\cong\mathscr{B}$. For dimensional reasons, we have $\rank\mathscr{B}'=1$. Choose a morphism $\eta:\Spec L\to\mathcal{Y}$, where $L$ is a field containing $E$, such that its image is dense in $\mathcal{Y}$. We can assume that $L$ is large enough so that $\eta$ has lifts to $\mathcal{R}^{\I}_\mathscr{B}$ and $\mathcal{R}^{\I}_{\mathscr{B}'}$. Then we have two exact sequences for the shtuka $\mathscr{F}$ over $\Spec L$
\[\begin{tikzcd}0\arrow[r]& \mathscr{S} \arrow[r] & \mathscr{F} \arrow[r] & \mathscr{Q} \arrow[r] & 0\end{tikzcd}\]
\[\begin{tikzcd}0\arrow[r]& \mathscr{S}'  \arrow[r] & \mathscr{F} \arrow[r] & \mathscr{Q}' \arrow[r] & 0\end{tikzcd}\]
where $\mathscr{Q}\in\mathcal{T}_{\mathscr{B}}(L)$ and $\mathscr{Q}'\in\mathcal{T}_{\mathscr{B}'}(L)$. Since $\mathscr{B}\not\cong\mathscr{B}'$, we have $\mathscr{S}\ne\mathscr{S}'$. Since $\mathscr{S}$ and $\mathscr{S}'$ are saturated in $\mathscr{F}$, so is $\mathscr{S}''=\mathscr{S}\cap\mathscr{S}'$. Thus $\rank \mathscr{S}''=\rank\mathscr{S}-1$. There is also an exact sequence of shtukas
\[\begin{tikzcd}0\arrow[r]& \mathscr{S}'' \arrow[r] & \mathscr{F} \arrow[r] & \mathscr{Q}'' \arrow[r] & 0\end{tikzcd}\]
where $\mathscr{Q}''$ is a trivial shtuka over $\Spec L$ of rank 2. This shows that the image of $\eta$ is contained in $\mathcal{Z}^{\I}_{\mathscr{B}''}$ for some $\mathscr{B}''\in\mathfrak{Vect}_X^2$. Since $\mathcal{Z}^{\I}_{\mathscr{B}''}$ is closed and the image of $\eta$ is dense in $\mathcal{Y}$, we see that $\mathcal{Y}$ is contained in $\mathcal{Z}^{\I}_{\mathscr{B}''}$. By \cref{codimension of horocycles of type I or type II}, $\mathcal{Y}$ has codimension 1 in $\mathcal{M}$, while $\mathcal{Z}^{\I}_{\mathscr{B}''}$ has codimension 2 in $\mathcal{M}$, a contradiction.
\end{proof}

The following statement is Proposition 1.1 of~\cite{Drinfeld1}.
\begin{Proposition}\label{descent of coherent sheaves for projective scheme over finite field}
  Let $S$ be a projective scheme over $\mathbb{F}_q$. Let $L$ be an algebraically closed field over $\mathbb{F}_q$. Then the functor $\mathscr{F}\mapsto\mathscr{F}\otimes L$ is an equivalence between the category of coherent sheaves $\mathscr{F}$ on $S$ and the category of coherent sheaves $\mathscr{M}$ on $S\otimes L$ equipped with an isomorphism $(\Id_S\otimes\Fr_L)^*\mathscr{M}\simto\mathscr{M}$. \qed
\end{Proposition}

\begin{Lemma}\label{non-containment for subshtuka of the same rank}
  Let $L$ be a field over $\mathbb{F}_q$. Let $\mathscr{F}$ be a shtuka over $\Spec L$ with zero $\alpha$ and pole $\beta$ satisfying condition (\ref{Frobenius shifts of zero and pole mutually disjoint}). Let $\mathscr{G}$ be a subshtuka of $\mathscr{F}$ of the same rank with the same zero and pole. Then $\Phi_L^*\mathscr{G}\not\subset\mathscr{F}$ and $\mathscr{G}\not\subset\Phi_L^*\mathscr{F}$.
\end{Lemma}
\begin{proof}
   We apply $d$-th exterior power to all sheaves involved to reduce the problem to the case $d=1$.

   Suppose $\Phi_L^*\mathscr{G}\subset \mathscr{F}$. Then we have $\mathscr{F}=\mathscr{G}(\beta+W)$ for some effective divisor $W$ of $X\otimes L$. From the isomorphisms $\Phi_L^*\mathscr{G}\cong\mathscr{G}(\beta-\alpha)$ and $\Phi_L^*\mathscr{F}\cong\mathscr{F}(\beta-\alpha)$ we deduce that $\beta-\alpha+\Fr_X\bcirc\beta+\Phi^*W=\beta+W+\beta-\alpha$. Hence $\beta+W=\Fr_X\bcirc\beta+\Phi^*W$. Applying \cref{difference between divisor and its Frobenius} to the two morphisms $\beta,\Fr_X\bcirc\beta:\Spec L\to X$, we see that $\beta=\Fr_X^i\bcirc\beta$ for some $i\ge 1$, a contradiction to condition (\ref{Frobenius shifts of zero and pole mutually disjoint}).

   The proof of the second statement is similar.
\end{proof}

\begin{Lemma}\label{subshtuka of the same rank}
  Let $L$ be an algebraically closed field. Let $\mathscr{F}$ be a shtuka of rank $d$ over $\Spec L$ with zero and pole satisfying condition (\ref{Frobenius shifts of zero and pole mutually disjoint}). Let $\mathscr{G}$ be a subshtuka of $\mathscr{F}$ of the same rank with the same zero and pole. Then $\mathscr{F}/\mathscr{G}$ is supported on $D\otimes L$ for some finite subscheme $D\subset X$. Moreover, for any structure of level $D$ on $\mathscr{F}$, $\mathscr{G}$ is obtained from $\mathscr{F}$ by applying Construction E in \cref{(Section)general constructions for shtukas} with respect to that level structure and an $\mathscr{O}_D$-submodule $\mathscr{R}\subset\mathscr{O}_D^d$.
\end{Lemma}
\begin{proof}
  Let $\mathscr{F}'=\Phi_L^*\mathscr{F}+\mathscr{F},\mathscr{G}'=\Phi_L^*\mathscr{G}+\mathscr{G}$. \cref{non-containment for subshtuka of the same rank} shows that $\mathscr{F}\cap\mathscr{G}'=\mathscr{G}$ and $\Phi_L^*\mathscr{F}\cap\mathscr{G}'=\Phi_L^*\mathscr{G}$. Thus the morphisms $\mathscr{F}/\mathscr{G}\to\mathscr{F}'/\mathscr{G}'$ and $\Phi_L^*(\mathscr{F}/\mathscr{G})\to\mathscr{F}'/\mathscr{G}'$ are injective. The sheaves $\mathscr{F}/\mathscr{G}, \mathscr{F}'/\mathscr{G}', \Phi_L^*(\mathscr{F}/\mathscr{G})$ are torsion sheaves on $X\otimes L$, and we have $h^0(\mathscr{F}/\mathscr{G})=h^0(\mathscr{F}'/\mathscr{G}')=h^0(\Phi_L^*(\mathscr{F}/\mathscr{G}))$. Hence the morphisms $\mathscr{F}/\mathscr{G}\to\mathscr{F}'/\mathscr{G}'$ and $\Phi_L^*(\mathscr{F}/\mathscr{G})\to\mathscr{F}'/\mathscr{G}'$ are isomorphisms. So we have $\Phi_L^*(\mathscr{F}/\mathscr{G})\cong\mathscr{F}/\mathscr{G}$.
  \cref{descent of coherent sheaves for projective scheme over finite field} gives an isomorphism $\mathscr{F}/\mathscr{G}\cong\mathscr{M}\otimes L$ for some coherent sheaf $\mathscr{M}$ on $X$. Since $\mathscr{F}$ and $\mathscr{G}$ have the same rank, $\mathscr{M}$ is supported on a finite subscheme $D\subset X$.

  Equip $\mathscr{F}$ with a structure of level $D$. Let $\mathscr{P}=\mathscr{G}/\mathscr{F}(-D\otimes L)\subset\mathscr{F}/\mathscr{F}(-D\otimes L)\cong\mathscr{O}_{D\otimes L}^d$. Since $\mathscr{F}(-D\otimes L)$ is a subshtuka of $\mathscr{G}$, we get an isomorphism $\Phi_L^*\mathscr{P}\simto\mathscr{P}$ which is compatible with the natural isomorphism $\Phi_L^*\mathscr{O}_{D\otimes L}^d\simto\mathscr{O}_{D\otimes L}^d$. \cref{descent of coherent sheaves for projective scheme over finite field} gives an isomorphism $\mathscr{P}\cong\mathscr{R}\otimes L$ for some $\mathscr{O}_D$-submodule $\mathscr{R}\subset\mathscr{O}_D^d$. We see that $\mathscr{G}$ is obtained from $\mathscr{F}$ by applying Construction E with respect to $\mathscr{R}$.
\end{proof}

\begin{Lemma}\label{horospherical divisor of type I not contained in horocycles of type II}
Let $\mathscr{B}\in\Pic(X)$. Then for any $1\le i\le d-1, \mathscr{A}\in\mathfrak{Vect}_X^i$, the horocycle $\mathcal{Z}^{\II}_{\mathscr{A}}$ does not contain any irreducible component of the horocycle $\mathcal{Z}^{\I}_\mathscr{B}$.
\end{Lemma}
\begin{proof}
Let $\mathcal{Y}$ be an irreducible component of $\mathcal{Z}^{\I}_\mathscr{B}$. Suppose $\mathcal{Y}$ is contained in $\mathcal{Z}^{\II}_{\mathscr{A}}$ for some locally free sheaf $\mathscr{A}$ on $X$ with $1\le\rank\mathscr{A}\le d-1$. For dimensional reasons, we have $\rank\mathscr{A}=1$. Choose a morphism $\eta:\Spec L\to\mathcal{Y}$, where $L$ is an algebraically closed field containing $E$, such that its image is dense in $\mathcal{Y}$. We can assume that $L$ is large enough so that $\eta$ has lifts to $\mathcal{R}^{\II}_\mathscr{A}$ and $\mathcal{R}^{\I}_{\mathscr{B}}$. Then we get two exact sequences for the shtuka $\mathscr{F}$ over $\Spec L$
\[\begin{tikzcd}
  0\arrow[r]&\mathscr{A}\otimes L\arrow[r]&\mathscr{F}\arrow[r]&\mathscr{Q}\arrow[r]&0
\end{tikzcd}\]
\[\begin{tikzcd}
  0\arrow[r]&\mathscr{S}\arrow[r]&\mathscr{F}\arrow[r]&\mathscr{B}\otimes L\arrow[r]&0
\end{tikzcd}\]
where $\mathscr{Q},\mathscr{S}\in\Sht_E^{d-1}(L)$.

By \cref{dimension of horocycles of type I+II}, $\mathcal{R}^{\I\wedge\II}_{\mathscr{A},\mathscr{B}}$ has dimension $2d-4$ over $\Spec E$, so its image in $\mathcal{M}$ does not contain the image of $\eta$. Therefore the composition $\mathscr{A}\otimes L\to\mathscr{F}\to\mathscr{B}\otimes L$ is nonzero, hence injective. So the morphism $\mathscr{F}\to\mathscr{Q}\oplus(\mathscr{B}\otimes L)$ is injective, and its cokernel is supported on $D\otimes L$ for some finite subscheme $D\subset X$ by \cref{subshtuka of the same rank}. Since $L$ is algebraically closed, we can equip $\mathscr{Q}\oplus(\mathscr{B}\otimes L)$ with a structure of level $D$. By \cref{subshtuka of the same rank}, $\mathscr{F}$ is obtained from $\mathscr{Q}\oplus(\mathscr{B}\otimes L)$ by applying construction E with respect to that level structure and an $\mathscr{O}_D$-submodule $\mathscr{R}\subset\mathscr{O}_D^d$.

Now we see that the image of $\eta$ is contained in the image of morphism $\Sht_E^{d-1}\to\Sht_E^d$ which sends a shtuka $\mathscr{G}$ of rank $d-1$ to the shtuka $\mathscr{G}\oplus(\mathscr{B}\otimes E)$. Since $\dim\Sht^{d-1}_E=2d-4=\dim\mathcal{Z}^{\I}_\mathscr{B}-1$, we get a contradiction.
\end{proof}

\begin{Lemma}\label{existence of very generic closed points on horospherical divisors}
Let $\mathscr{B}\in\Pic(X)$. Let $\mathsf{q}:\mathcal{M}'\to\mathcal{M}$ be a surjective \'etale morphism where $\mathcal{M}'$ is a scheme. Choose $\mathcal{W}$ to be an irreducible component of $\mathsf{q}^{-1}(\mathcal{Z}^{\I}_\mathscr{B})$.

Then for any open dense subscheme $\mathcal{U}$ of $\mathcal{W}$, we can find an extension field $F$ containing $E$ such that there exists an $F$-point on $\mathcal{U}\otimes_E F$ whose image in $\mathcal{M}\otimes_E F$ is contained in the image of the morphism $\mathcal{R}^{\I}_\mathscr{B}\otimes_E F\to\mathcal{M}\otimes_E F$, is not contained in $\mathcal{Z}^{\II}_\mathscr{A}\otimes_E F$ for any nonzero locally free sheaf $\mathscr{A}$ on $X$ with $1\le\rank\mathscr{A}\le d-1$, and is not contained in $\mathcal{Z}^{\I}_{\mathscr{B}'}\otimes_E F$ for any nonzero locally free sheaf $\mathscr{B}'$ on $X$ not isomorphic to $\mathscr{B}$ with $1\le\rank\mathscr{B}'\le d-1$.
\end{Lemma}
\begin{proof}
\cref{horospherical divisor of type I not contained in other horocycles of type I} and \cref{horospherical divisor of type I not contained in horocycles of type II} show that the statement is true if $F$ is the field of rational functions of $\mathcal{W}$.
\end{proof}

\subsection{Definition of Multiplicity}\label{section: definition of multiplicity}
Let $\pi:X\times\mathcal{M}\to X$ be the projection.

For a coherent sheaf $\mathscr{E}$ on $\mathcal{M}$ which is zero outside horospherical divisors, we can define its multiplicity at horospherical divisors as follows. Since $\mathcal{M}$ is a Deligne-Mumford stack by Corollary 6 of Section 3 of Chapter of~\cite{L.Lafforgue97}, we can choose an \'etale surjective morphism $\mathsf{q}:\mathcal{M}'\to\mathcal{M}$ where $\mathcal{M}'$ is a scheme. Choose $\mathcal{Y}$ to be an irreducible component of a horospherical divisor, and $\mathcal{W}$ to be an irreducible component of the inverse image of $\mathcal{Y}$ in $\mathcal{M}'$. Let $\eta$ denote the generic point of $\mathcal{W}$. The local ring $\mathscr{O}_{\mathcal{M}',\eta}$ is a discrete valuation ring and $\mathsf{q}^*\mathscr{E}$ restricts to a torsion $\mathscr{O}_{\mathcal{M}',\eta}$-module of finite length. We see that the length does not depend on the choice of $\mathcal{M}'$ and $\mathcal{W}$, and we define it to be the multiplicity of $\mathscr{E}$ at $\mathcal{Y}$, denoted by $\mult_\mathcal{Y}\mathscr{E}$.

By \cref{cohomology of irreducible shtukas}, the zero morphism $0\to R\pi_*\mathscr{F}$ is a quasi-isomorphism on the complement of the horocycles. Since horocycles have positive codimension, by Theorem 3 of~\cite{KM}, we get a canonical isomorphism $\det (R\pi_*\mathscr{F})\xrightarrow{\sim}\mathscr{O}_\mathcal{M}(-\mathcal{D})$, where
\[\mathcal{D}=-\sum_{\mathcal{Y}}(\mult_{\mathcal{Y}}(\pi_*\mathscr{F})-\mult_{\mathcal{Y}}(R^1\pi_*\mathscr{F}))\]
where $\mathcal{Y}$ runs through all horospherical divisors.
Since $\mathscr{F}$ is locally free, its direct image $\pi_*\mathscr{F}$ is torsion free. By \cref{cohomology of irreducible shtukas}, $\pi_*\mathscr{F}$ is zero on the complement of horocycles. Hence $\pi_*\mathscr{F}=0$. So we get
\[\mathcal{D}=\sum_{\mathcal{Y}}\mult_{\mathcal{Y}}(R^1\pi_*\mathscr{F})\]

\subsection{Mittag-Leffler condition for automorphisms of modules}
\begin{Definition}
Let $(S_n)_{n\ge 1}$ be an inverse system of sets with respect to the natural order on positive integers. We say that it satisfies the \emph{Mittag-Leffler condition} if for each for each $n\ge 1$, there exists $N\ge n$ such that the images of $S_i$ and $S_j$ in $S_n$ coincide for all $i,j\ge N$.
\end{Definition}

\begin{Lemma}
  If a inverse system of sets $(S_n)_{n\ge 1}$ satisfies the Mittag-Leffler condition, and each $S_n$ is nonempty, then $\varprojlim_n S_n$ is nonempty. \qed
\end{Lemma}

Let $A$ be a complete Noetherian local ring with maximal ideal $\mathfrak{m}$. Put $A_n=A/\mathfrak{m}^n$. Let $M$ be a finitely generated $A$-module and $M_n=M\otimes_A A_n$. Let $G_n$ be the group of automorphisms of $M_n$. Then we have natural group homomorphisms $G_{n+1}\to G_n$ for all $n\ge 1$.

\begin{Lemma}\label{Mittag-Leffler condition for automorphisms of modules}
The inverse system $(G_n)_{n\ge 1}$ satisfies the Mittag-Leffler condition. In particular, $\varprojlim_n G_n$ is nonempty.
\end{Lemma}
\begin{proof}
Let $E_n$ be the endomorphism algebra of $M_n$. We have natural homomorphisms $E_{n+1}\to E_n$ for all $n\ge 1$. Since $A$ is Noetherian and $M_n$ is finitely generated over $A$, $E_n$ has finite dimension over $A/\mathfrak{m}$. Hence $E_n$ is an Artinian $A_n$-module. So the inverse system $(E_n)_{n\ge 1}$ satisfies the Mittag-Leffler condition. For each $n$, let $E'_n$ be the image of $E_m$ in $E_n$ for sufficiently large $m$. We see that $E'_n$ is a subalgebra of $E_n$, and the natural homomorphisms $E'_{n+1}\to E'_n$ are surjective for all $n\ge 1$.

Let $U_n$ be the group of units of $E'_n$. Then $U_n$ is a subgroup of $G_n$, and the image of $G_m$ in $G_n$ is contained in $U_n$ for sufficiently large $m$.

Now it suffices to show that the homomorphism $U_{n+1}\to U_n$ is surjective for each $n\ge 1$. Let $u_n\in U_n$. We can find $v_n\in U_n$ such that $u_nv_n=1$. Since $E'_{n+1}\to E'_n$ is surjective, we can choose lifts $u_{n+1}, v_{n+1}\in E'_{n+1}$ of $u_n,v_n$. Then $f=u_{n+1}v_{n+1}-1$ satisfies $f(M_{n+1})\subset\mathfrak{m}^nM_{n+1}$. Thus $f^2(M_{n+1})\subset\mathfrak{m}^{2n}M_{n+1}=0$. Hence $u_{n+1}v_{n+1}(1-f)=1$. This shows that $u_{n+1}$ has a right inverse. Similarly, we can show that $u_{n+1}$ has a left inverse. Thus $u_{n+1}\in U_{n+1}$.
\end{proof}

\subsection{Multiplicity at Horospherical Divisors}
Choose a surjective \'etale morphism $\mathsf{q}:\mathcal{M}'\to\mathcal{M}$ where $\mathcal{M}'$ is a scheme. Let $\pi':X\times\mathcal{M}'\to\mathcal{M}'$ be the projection. Let $\mathscr{F}'$ be the shtuka over $\mathcal{M}'$.

Let $\mathscr{A},\mathscr{B}\in\Pic(X)$. Let $\mathcal{V}$ be an irreducible component of $\mathsf{q}^{-1}(\mathcal{Z}^{\II}_\mathscr{A})$. Let $\xi$ be the generic point of $\mathcal{V}$. Let $\mathcal{W}$ be an irreducible component of $\mathsf{q}^{-1}(\mathcal{Z}^{\I}_\mathscr{B})$. Let $\eta$ be the generic point of $\mathcal{V}$.

\begin{Proposition}\label{H^1 at horospherical divisors}
We have isomorphisms
\[(R^1\pi'_*\mathscr{F}')_\xi\cong \bigoplus_{i=0}^{h^0(\mathscr{A})-1}\mathscr{O}_{\mathcal{M}',\xi}/\mathfrak{m}_\xi^{q^i}\]
\[(R^1\pi'_*\mathscr{F}')_\eta\cong \bigoplus_{i=0}^{h^1(\mathscr{B})-1}\mathscr{O}_{\mathcal{M}',\eta}/\mathfrak{m}_\eta^{q^i}\]
where $(R^1\pi'_*\mathscr{F}')_\xi$ is the stalk of $R^1\pi'_*\mathscr{F}'$ at $\xi$, $\mathfrak{m}_\xi$ is the ideal of the restriction of $\mathcal{V}$ to $\Spec\mathscr{O}_{\mathcal{M}',\xi}$, $(R^1\pi'_*\mathscr{F}')_\eta$ is the stalk of $R^1\pi'_*\mathscr{F}'$ at $\eta$, $\mathfrak{m}_\eta$ is the ideal of the restriction of $\mathcal{W}$ to $\Spec\mathscr{O}_{\mathcal{M}',\eta}$.
\end{Proposition}

\begin{Theorem}\label{divisor of determinant of cohomology of shtukas}
We have a canonical isomorphism $\det (R\pi_*\mathscr{F}_0)\xrightarrow{\sim} \mathscr{O}_\mathcal{M}(-\mathcal{D})$, where
\[\mathcal{D}=\sum_{[\mathscr{B}]\in \Pic X}\frac{q^{h^1(\mathscr{B})}-1}{q-1}\cdot\mathcal{Z}^{\I}_\mathscr{B}
+\sum_{[\mathscr{A}]\in\Pic X}\frac{q^{h^0(\mathscr{A})}-1}{q-1}\cdot\mathcal{Z}^{\II}_\mathscr{A}\]
\end{Theorem}
\begin{proof}
In \cref{section: definition of multiplicity}, the multiplicity of $\mathcal{D}$ at $\mathcal{Z}^{\I}_\mathscr{B}$ is shown to be the multiplicity of $R^1\pi'_*\mathscr{F}'$ at $\mathcal{W}$, which by \cref{H^1 at horospherical divisors} is equal to
\[\sum_{i=0}^{h^1(\mathscr{B})-1}q^i=\frac{q^{h^1(\mathscr{B})}-1}{q-1}\]
Similarly, we get the multiplicity of $\mathcal{D}$ at $\mathcal{Z}^{\II}_\mathscr{A}$.
\end{proof}

\begin{proof}[Proof of \cref{H^1 at horospherical divisors}]
The generic point of $\Spec \mathscr{O}_{\mathcal{M}',\eta}$ is not contained in any horocycles. Hence $(R^1\pi'_*\mathscr{F}')_\eta$ is a torsion module over $\mathscr{O}_{\mathcal{M}',\eta}$ by \cref{cohomology of irreducible shtukas}. It is finitely generated since $\mathscr{F}'$ is coherent. Since $\mathcal{W}$ has codimension 1 in $\mathcal{M}'$, $\mathscr{O}_{\mathcal{M}',\eta}$ is a discrete valuation ring with maximal ideal $\mathfrak{m}_\eta$. We can thus choose an isomorphism
\[(R^1\pi'_*\mathscr{F}')_\eta\cong \bigoplus_{i=0}^{h-1} \mathscr{O}_{\mathcal{M}',\eta}/\mathfrak{m}_\eta^{m_i}\]
where $m_i>0$ for $i=1,2,\dots,h$. So there exists an open dense subscheme $\mathcal{U}$ of $\mathcal{W}$ such that for any $w\in\mathcal{U}$, we have an isomorphism
\begin{equation}\label{isomorphism 1 for H^1 at a general point of horospherical divisor}
R^1\pi'_*\mathscr{F}'\otimes_{\mathscr{O}_{\mathcal{M}'}}A_w\cong \bigoplus_{i=0}^{h-1}A_w/\mathfrak{p}_w^{m_i}
\end{equation}
where $A_w$ is the completion of the local ring of $\mathcal{M}'$ at $w$, and $\mathfrak{p}_w$ is the ideal of the restriction of $\mathcal{W}$ to $\Spec A_w$.

The statement is true if it is true when $E$ is replaced by an extension. So we can assume $F=E$ in \cref{existence of very generic closed points on horospherical divisors}. Thus we can pick an $E$-point $w$ in $\mathcal{U}$ such that the shtuka $\mathscr{F}_w$ over $w$ satisfies $\mathscr{F}_w/(\mathscr{F}_w)^{\I}\cong \mathscr{B}\otimes E$ and $(\mathscr{F}_w)^{\II}=0$, where $(\mathscr{F}_w)^{\I}$ and $(\mathscr{F}_w)^{\II}$ are as in \cref{maximal trivial sub and maximal trivial quotient}, i.e., $(\mathscr{F}_w)^{\II}$ is the maximal trivial subshtuka and $\mathscr{F}_w/(\mathscr{F}_w)^{\I}$ is the maximal trivial quotient shtuka.

Let $\mathfrak{m}_w$ be the maximal ideal of $A_w$. Put $A_{w,n}=A_w/\mathfrak{m}_w^n$. \cref{deformation in both directions} shows that when $n\ge q^{h^1(\mathscr{B})}$, there exists an isomorphism
\[R^1\pi'_*\mathscr{F}'\otimes_{\mathscr{O}_{\mathcal{M}'}}A_{w,n}\cong \bigoplus_{i=0}^{h^1(\mathscr{B})-1}A_{w,n}/\mathfrak{a}_{w,n}^{q^i}\]
where $\mathfrak{a}_{w,n}$ is the ideal of $A_{w,n}$ corresponding to $H^\sharp\times P^\flat$ in \emph{loc. cit}. By \cref{Mittag-Leffler condition for automorphisms of modules}, we can choose the above isomorphisms so that they are compatible with respect to the homomorphisms $A_{w,n+1}\twoheadrightarrow A_{w,n}$. Now Artin-Rees lemma implies that
\begin{equation}\label{isomorphism 2 for H^1 at a general point of horospherical divisor}
R^1\pi'_*\mathscr{F}'\otimes_{\mathscr{O}_{\mathcal{M}'}}A_w\cong \bigoplus_{i=0}^{h^1(\mathscr{B})-1}A_w/\mathfrak{a}_w^{q^i}
\end{equation}
where $\mathfrak{a}_w$ is an ideal of $A_w$.

Comparing isomorphisms (\ref{isomorphism 1 for H^1 at a general point of horospherical divisor}) and (\ref{isomorphism 2 for H^1 at a general point of horospherical divisor}), we conclude that $h=h^1(\mathscr{B})$ and $m_i=q^i$ up to a permutation of $i$. This gives the second isomorphism in the statement.

The first isomorphism follows from the second one and Grothendieck-Serre duality.
\end{proof}

\section{Construction of rational functions on the moduli scheme of shtukas with all level structures}
Let $\Sht_{\all}^d$ be the moduli scheme of shtukas of rank $d$ equipped with structures of all levels compatible with each other.

In this section, we review the construction of the action of $GL_d(\mathbb{A})/k^\times$ on $\Sht_{\all}^d$ and the Tate central extension of $GL_d(\mathbb{A})$. Then we use them to obtain equivariant structures on the determinant line bundle. Finally, in \cref{rational functions and principal divisors on the moduli scheme of shtukas}, we  construct some rational functions and obtain some principal divisors on the moduli scheme of shtukas with all level structures.

\subsection{The action of $GL_d(\mathbb{A})/k^\times$ on $\Sht_{\all}^d$}\label{Section: adelic action on the moduli scheme of shtukas}
We recall the (left) action of $GL_d(\mathbb{A})/k^\times$ on $\Sht_{\all}^d$ defined in Section 3 of~\cite{Drinfeld1}.

We first define an action of the semigroup $GL_d^-(\mathbb{A})=\{g\in GL_d(\mathbb{A})|g^{-1}\in\Mat_{d\times d}(O)\}$.

Let $\mathscr{F}\in\Sht_{\all}^d(S)$ and let $g\in GL_d^-(\mathbb{A})$. We have $g^{-1}O^d\subset O^d$. Choose an open ideal $I\subset O$ such that $I\cdot O^d\subset g^{-1}O^d$. Let $D\subset X$ be the subscheme corresponding to $I$, and let $\mathscr{R}\subset\mathscr{O}_D^d$ be the $\mathscr{O}_D$-submodule corresponding to the submodule $g^{-1}O^d/I\cdot O^d\subset(O/I)^d$. Applying Construction E in \cref{(Section)general constructions for shtukas} to $\mathscr{F}$ and $\mathscr{R}$, we obtain a right shtuka $\mathscr{F}'$, which does not depend on the choice of $I$.

Let $D'$ be a finite subscheme of $X$ and let $J\subset O$ be the ideal corresponding to $D'$. In order to equip $\mathscr{F}$ with a structure of level $D'$, we choose $I$ such that $I\cdot O^d\subset J\cdot h^{-1}O^d$. Then the composition $g^{-1}O^d/I\cdot O^d\to g^{-1}O^d/J\cdot g^{-1}O^d\xrightarrow{g}(O/J)^d$ is surjective. Applying Construction E', we obtain a structure of level $D'$ on $\mathscr{F}'$.

It is easy to see that the above construction induces an action of the semigroup $GL_d^-(\mathbb{A})$ on $\Sht_{(\alpha,\beta),\all}^{d}$.

On the other hand, we have a canonical isomorphism $\mathbb{A}^\times/k^\times=\varprojlim\limits_{D}\Pic_D$. Construction D' gives the action of $\Pic_D(X)$ on $\Sht_{\all}^d$. Passing to projective limit, we get an action of $\mathbb{A}^\times/k^\times$ on  $\Sht_{\all}^d$. One can check that this action on $\mathbb{A}^\times\cap GL_d^-(\mathbb{A})$ coincides with the restriction of the action of $GL_d^-(\mathbb{A})$.

Since $GL_d(\mathbb{A})=\mathbb{A}^\times\cdot GL_d^-(\mathbb{A})$, we obtain an action of $GL_d(\mathbb{A})/k^\times$ on $\Sht_{\all}^d$.

\subsection{A canonical central extension of $GL_d(\mathbb{A})$}
We borrow the definition of relative determinant from Section 4 of~\cite{AdCK}. In particular, for any $g_1,g_2\in GL_d(\mathbb{A})$, we define
\[\det\nolimits_{g_1O^d}^{g_2O^d}:=\det(g_1O^d/g_1O^d\cap g_2O^d)^*\otimes\det(g_2O^d/g_1O^d\cap g_2O^d).\]
For $g_1,g_2,g_3\in GL_d(\mathbb{A})$, we have a canonical isomorphism
\[\det\nolimits_{g_1O^d}^{g_2O^d}\otimes\det\nolimits_{g_2O^d}^{g_3O^d}\cong\det\nolimits_{g_1O^d}^{g_3O^d}.\]

As in Section 5 of~\cite{AdCK}, we define a central extension
\[\begin{tikzcd}
1 \arrow[r]& \mathbb{F}_q^\times \arrow[r]& \widehat{GL_d}(\mathbb{A}) \arrow[r]& GL_d(\mathbb{A}) \arrow[r]& 1.
\end{tikzcd}\]
Here $\widehat{GL_d}(\mathbb{A})=\{(g,a)\in GL_d(\mathbb{A})\times\det_{O^d}^{g O^d}|a\ne 0\}$. The group structure on $\widehat{GL_d}(\mathbb{A})$ is given by  $(g,a)(h,b)=(gh,a\cdot(gb))$. Note that $a\in\det_{O^d}^{gO^d}$ and $gb\in\det_{gO^d}^{ghO^d}$, so $a\cdot(gb)\in\det_{O^d}^{gO^d}\otimes\det_{gO^d}^{ghO^d}\cong\det_{O^d}^{ghO^d}$.

\subsection{A $\widehat{GL_d}(\mathbb{A})$-equivariant structure on the determinant of cohomology of shtukas}\label{(Section)equivariant structure on the determinant of cohomology of shtukas}

Let $\mathscr{F}_{\all}$ be the universal shtuka over $\Sht_{\all}^d$.

Let $\widehat{GL_d}(\mathbb{A})$ act on $\Sht_{\all}^d$ through its quotient $GL_d(\mathbb{A})/k^\times$.

We define a $\widehat{GL_d}(\mathbb{A})$-equivariant structure on $\mathscr{L}:=\det R\pi_*\mathscr{F}_{\all}$. In other words, for each $u\in\widehat{GL_d}(\mathbb{A})$ we define an isomorphism $\phi_u:u^*\mathscr{L}\xrightarrow{\sim}\mathscr{L}$ such that the cocycle condition $\phi_v\bcirc v^*(\phi_u)=\phi_{uv}$ holds for any $u,v\in\widehat{GL_d}(\mathbb{A})$.

First we define a group
\[\widetilde{GL_d}(\mathbb{A})=\{(g,\phi)\in GL_d(\mathbb{A})\times\Isom(g^*\mathscr{L},\mathscr{L})\}\]
The group structure is given by $(g,\phi)(h,\psi)=(gh,\psi\bcirc h^*(\phi))$. We have a projection homomorphism $\widetilde{GL_d}(\mathbb{A})\to GL_d(\mathbb{A})$.

Note that a $\widehat{GL_d}(\mathbb{A})$-equivariant structure on $\mathscr{L}$ is the same as a homomorphism $f:\widehat{GL_d}(\mathbb{A})\to\widetilde{GL_d}(\mathbb{A})$ such that the following diagram commutes.
\begin{equation}\label{diagram for an equivalent definition of equivariant structure}
\begin{tikzcd}[column sep=tiny]
  \widehat{GL_d}(\mathbb{A})\arrow[rr,"f"]\arrow[rd] && \widetilde{GL_d}(\mathbb{A})\arrow[ld]\\
  &GL_d(\mathbb{A})
\end{tikzcd}
\end{equation}

Let $\widehat{GL_d^-}(\mathbb{A})$ be the preimage of $GL_d^-(\mathbb{A})$ in $\widehat{GL_d}(\mathbb{A})$. We first construct a homomorphism $f^-:\widehat{GL_d^-}(\mathbb{A})\to\widetilde{GL_d}(\mathbb{A})$.

For any $g\in GL_d^-(\mathbb{A})$, from the definition of the action of $g$ in \cref{Section: adelic action on the moduli scheme of shtukas}, we get an exact sequence of coherent sheaves on $X\times\Sht_{\all}^d$
\[\begin{tikzcd}
0 \arrow[r]& g^*\mathscr{F}_{\all} \arrow[r]& \mathscr{F}_{\all} \arrow[r]& (O^d/g^{-1}O^d)\boxtimes\mathscr{O}_{\Sht_{\all}^d} \arrow[r]& 0,
\end{tikzcd}\]
which induces an isomorphism
\[g^*\mathscr{L}\otimes\det\nolimits_{g^{-1}O^d}^{O^d}\cong\mathscr{L}\]
of invertible sheaves on $\Sht_{\all}^d$. Thus we have an inclusion
\[(\det\nolimits_{g^{-1}O^d}^{O^d}-\{0\})\subset\Isom(g^*\mathscr{L},\mathscr{L}).\]
Moreover, for any $g,h\in GL_d^-(\mathbb{A})$, the natural isomorphism
\[\det\nolimits_{h^{-1}g^{-1}O^d}^{h^{-1}O^d}\otimes\det\nolimits_{h^{-1}O^d}^{O^d}\cong\det\nolimits_{h^{-1}g^{-1}O^d}^{O^d}\]
is compatible with the composition
\[\Isom(h^*g^*\mathscr{L},h^*\mathscr{L})\times \Isom(h^*\mathscr{L},\mathscr{L})\to\Isom(h^*g^*\mathscr{L},\mathscr{L})\]

Now we define $f^-:\widehat{GL_d^-}(\mathbb{A})\to\widetilde{GL_d}(\mathbb{A})$. For $u=(g,a)\in\widehat{GL_d^-}(\mathbb{A})$, we define $f^-(u)$ to be $(g,\phi_u)$, where $\phi_u:g^*\mathscr{F}_{\all}\xrightarrow{\sim}\mathscr{F}_{\all}$ is $g^{-1}a\in\det_{g^{-1}O^d}^{O^d}$. The above discussion shows that $f^-$ is a homomorphism.

For any $w\in\widehat{GL_d}(\mathbb{A})$ there exists $u_1,v_1,u_2,v_2\in\widehat{GL_d^-}(\mathbb{A})$ such that $w=u_1v_1^{-1}=v_2^{-1}u_2$. Thus $f$ extends uniquely to a homomorphism $f:\widehat{GL_d}(\mathbb{A})\to\widetilde{GL_d}(\mathbb{A})$ by the following lemma. It is easy to see that $f$ makes diagram (\ref{diagram for an equivalent definition of equivariant structure}) commute.

\begin{Lemma}
  Let $G$ be a group and let $G^-$ be a subset $G$ which forms a semigroup under the group structure of $G$. Suppose that for any $g\in G$, there exists $u_1,v_1,u_2,v_2\in G^-$ such that $g=u_1v_1^{-1}=v_2^{-1}u_2$. Then any homomorphism from $G^-$ to a group $H$ extends uniquely to a homomorphism from $G$ to $H$. \qed
\end{Lemma}

\subsection{Rational functions}
Let $GL_d(\mathbb{A})_0=\{g\in GL_d(\mathbb{A})|\deg g=0\}$. Let $\widehat{GL_d}(\mathbb{A})_0$ be the preimage of $GL_d(\mathbb{A})_0$ under the homomorphism $\widehat{GL_d}(\mathbb{A})\to GL_d(\mathbb{A})$.

Let $E$ be an algebraically closed field and let $\alpha,\beta:\Spec E\to X$ be two morphisms satisfying condition (\ref{Frobenius shifts of zero and pole mutually disjoint}).

Let $\mathcal{M}$  denote the moduli stack of shtukas $\mathscr{F}$ of rank $d$ over $\Spec E$ with zero $\alpha$ and pole $\beta$ satisfying $\chi(\mathscr{F})=0$. Let $\mathcal{M}_{\all}$ denote the moduli scheme of shtukas satisfying the same condition as $\mathcal{M}$ and equipped with structures of all levels compatible with each other.

Let $\mathscr{F}_{0,\all}$ denote the universal shtuka over $\mathcal{M}_{\all}$. Let $\pi:X\times\mathcal{M}_{\all}\to\mathcal{M}_{\all}$ be the projection. Put $\mathscr{L}_0=\det R\pi_*\mathscr{F}_{0,\all}$.

As in \cref{divisor of determinant of cohomology of shtukas}, let
\[\mathcal{D}=\sum_{[\mathscr{B}]\in \Pic X}\frac{q^{h^1(\mathscr{B})}-1}{q-1}\cdot\mathcal{Z}^{\I}_\mathscr{B}
+\sum_{[\mathscr{A}]\in\Pic X}\frac{q^{h^0(\mathscr{A})}-1}{q-1}\cdot\mathcal{Z}^{\II}_\mathscr{A}\]
be the divisor of $\mathcal{M}$ . Let $\mathcal{D}_{\all}$ be the pullback of $\mathcal{D}$ from $\mathcal{M}$ to $\mathcal{M}_{\all}$.

\begin{Proposition}
  We have a canonical isomorphism $\mathscr{L}_0\cong\mathscr{O}_{\mathcal{M}_{\all}}(-\mathcal{D}_{\all})$.
\end{Proposition}
\begin{proof}
  The statement follows from \cref{divisor of determinant of cohomology of shtukas} and functoriality of $\det$ in~\cite{KM}.
\end{proof}

The $\widehat{GL_d}(\mathbb{A})$-equivariant structure in \cref{(Section)equivariant structure on the determinant of cohomology of shtukas} induces a $\widehat{GL_d}(\mathbb{A})_0$-equivariant structure on $\mathscr{L}_0$. We obtain the following theorem.

\begin{Theorem}\label{rational functions and principal divisors on the moduli scheme of shtukas}
  For any two elements $u,v\in\widehat{GL_d}(\mathbb{A})_0$, there is a canonical isomorphism $u^*\mathscr{L}_0\simto v^*\mathscr{L}_0$. In particular, the divisor $g^*\mathcal{D}_{\all}-h^*\mathcal{D}_{\all}$ is principal for any $g,h\in GL_d(\mathbb{A})_0$. \qed
\end{Theorem}

\bibliography{../../../Bibliography/Reference}
\bibliographystyle{amsplain}
\end{document}